\documentclass[10pt]{article}
\usepackage{cppmacros}
\graphicspath{{./figs/}}
\numberwithin{equation}{section}
\usepackage{natbib}

\title{Finite volume methods for the computation of statistical solutions of the incompressible Euler equations}
\author{Carlos Par\'es-Pulido\footnote{\textit{carlos.pares-pulido@sam.math.ethz.ch}, R\"amistrasse 101, Z\"urich, Switzerland}, Seminar for Applied Mathematics, ETH Zurich}

\begin{document}

\maketitle

\begin{abstract}
{We present an efficient numerical scheme based on Monte Carlo integration to approximate statistical solutions of the incompressible Euler equations. The scheme is based on finite volume methods, which provide a more flexible framework than previously existing spectral methods for the computation of statistical solutions for incompressible flows. This finite volume scheme is rigorously proven to, under experimentally verifiable assumptions, converge in an appropriate topology and with increasing resolution to a statistical solution. The convergence obtained is stronger than that of measure-valued solutions, as it implies convergence of multi-point correlation marginals. We present results of numerical experiments which support the claim that the aforementioned assumptions are very natural, and appear to hold in practice.
}
\end{abstract}

\section{Introduction}
The formal limit of the Navier--Stokes equations as the \emph{Reynolds number} tends to infinity produces the \emph{incompressible Euler equations}, classically written as
\begin{equation*}
\begin{cases}
\partial_t \bu + \bu \cdot \nabla \bu + \nabla p = 0,\\
\div( \bu) = 0, \\
\bu|_{t=0} = \barbu;
\end{cases}
\end{equation*}
where $\bu \in \R^d$ ($d \in \{2, 3\}$) is the fluid velocity, $p \in [0, \infty)$ is the pressure, and $\barbu$ is a known initial condition in some function space to be detailed. We will consider these equations defined on a finite time domain $[0, T]$ for $T>0$, and a spatial domain $D$ which we will always take to have periodic boundary conditions; i.e., $D = \T^d$. These equations are widely used to model turbulent flows, in which the Reynolds number is large.

\subsection{Existence and uniqueness of solutions}
The problem of existence and uniqueness of classical (i.e. differentiable) solutions of the equations of motion for incompressible fluids is, in a general setting, an open question; cf. \cite{FeffermanClay}. For the incompressible Euler equations, in two dimensions and for smooth initial data, it is known that classical solutions exist globally, \cite{Ladyzhenskaya69,BealeKatoMajda}. In three dimensions, however, the only existence results available are local in time, \cite{Lichtenstein25}. Furthermore, the study of non-differentiable initial data (e.g. $\barbu \in L^2(D)$) is often of interest. Therefore one often studies the \emph{weak formulation} of the incompressible Euler equations:

\begin{definition} \label{def:weaksol}
	A vector field $\bu \in L^\infty([0, T); L^2(D, \R^d))$ is a \emph{weak solution of the incompressible Euler equations} with initial datum $\barbu \in L^2(D; \R^d)$, if for all test vector fields $\phi \in C_c^\infty([0, T) \times D; \R^d)$ with $\div(\phi) = 0$, it holds that
	\begin{equation*}
	\int_0^T \int_{D} \left[\bu \cdot \partial_t \phi + (\bu \otimes \bu) \colon \nabla \phi \right]\, dxdt = - \int_{D} \barbu \cdot \phi(0,x) \,dx,
	\end{equation*}
	and for all test functions $\psi \in C^\infty(D)$,
	\begin{equation*}
	\int_{D} \bu \cdot \nabla \psi \, dx = 0.
	\end{equation*}
\end{definition}

In the equations above, $\otimes$ denotes the tensor product of vectors, and $\colon$ the inner product for matrices.

Existence of weak solutions, globally in time, is a classical result, \cite{Leray34}. Let us recall the concept of vorticity, $\bvort \coloneqq \curl(\bu) \coloneqq \nabla \times \bu$.
In two dimensions, global existence and uniqueness of solutions holds if the initial vorticity $\bar{\bvort} \coloneqq \curl(\barbu)$ verifies $\bar{\bvort} \in L^\infty(\R^2)$, \cite{Yudovich63}, or $\bar{\bvort} \in L^1(\R^2) \cap L^p(\R^2)$, with $1 < p < \infty$, \cite{DiPerna87}. Furthermore, global existence of solutions (without uniqueness) is known if the initial vorticity $\bar{\bvort} \in H^{-1}(\R^2)$ is of the form $\bar{\bvort} = \mu + g$, with $\mu \in \mathcal{M}_+$ a positive Radon measure, and $g \in L^p$, for $p \in [1, \infty]$, \cite{Delort91,Vecchi93}. In the torus, in two and three dimensions, if $\barbu \in L^2(\T^d)$, there exist global solutions (in fact, infinitely many) with bounded kinetic energy, \cite{WiedemannExistence}.

\cite{Chen2012} proved that under an assumption, informed by Kolmogorov turbulence theory, any vanishing-viscosity sequence of weak solutions of the Navier-Stokes equations in $\T^3$, up to a sub-sequence, converges strongly to a weak solution of the incompressible Euler equations. We remark that we will use in the sequel, cf. \eqref{hyp:scaling}, an assumption analogous to this. However, this result does not guarantee that the limit solution to incompressible Euler is unique.

The results above are representative of the difficulties of uniqueness for weak solutions. In fact, there exists a dense set in $L^2(D; \R^d)$ of so-called \emph{wild initial data}, each of which admits infinitely many weak solutions; \cite{DeLellisSz09,DeLellisSz2013}. The discontinuous shear layer, which we discuss in detail in Section \ref{ss:sls}, is known to belong to this set, \cite{Szekelyhidi11}. Even restricting Definition \ref{def:weaksol} to functions in the class of \emph{admissible weak solutions}, i.e. with non-increasing kinetic energy $\|\bu\|_{L^2}^2$, uniqueness in a general setting fails. 

We remark here the following recent result, due to \cite{LMP1}: a weak solution to the incompressible Euler equations is dubbed \emph{physically realizable} if it is the weak limit in $L^2$ of a vanishing viscosity sequence of solutions of the incompressible Navier--Stokes equations. In \emph{two spatial dimensions}, physically realizable solutions conserve kinetic energy if and only if they are the \emph{strong limit} in $L^p([0, T]; L^2(D; \R^d))$, $p \ge 1$, of some subsequence of the vanishing viscosity sequence. As wild initial data admit weak solutions which do not conserve energy, this suggests that the aforementioned class might not be physically relevant (in the sense of vanishing viscosity). In three dimensions, however, no analogous result is available, as there exist energy-dissipative solutions which are realizable as a strong vanishing viscosity limit of weak solutions to incompressible Navier--Stokes, \cite{Buckmaster2019}.

There exist many numerical schemes for the computation of approximate solutions of the incompressible equations of fluid dynamics; a survey can be found e.g. in \cite{Langtangen2002}. Of particular interest for this work is the finite volume scheme of \cite{BCG89}, based on a discrete Leray projection. Note that the lack of uniqueness discussed above hinders the application of numerics, regardless of the scheme. If no unique solution exists, numerical schemes may not produce a Cauchy sequence as the resolution increases; this is readily observable with the Bell--Colella--Glaz scheme for the Euler equations, as acknowledged in the original description of the method. Even for schemes that are provably convergent under sufficient assumptions of regularity, see e.g. \cite{MajdaBertozzi}, this convergence is very slow, limiting the practical usefulness of the results.

\subsection{Expanded solution spaces for the incompressible Euler equations}
The absence of well-posedness results for weak solutions in a general setting, as well as the difficulties for numerical approximations, leave the door open for defining alternative frameworks for solutions. One which has proven popular is that of \emph{measure-valued solutions}, originally presented in \cite{DiPerna85,DiPerna87}, in which the solution is searched for in the space of probability measures parameterized in space and time, i.e. \emph{Young measures}.

Measure-valued solutions of the incompressible Euler equations are known to exist globally, \cite{DiPerna87}, and there exist efficient algorithms to approximate them (in weak* topology), e.g. via Monte Carlo integration, e.g. \cite{LM15,Leonardi18}, or by C\'esaro-type averaging of approximations at different resolutions, e.g. \cite{Feireisl19}.

Furthermore, the following \emph{weak-strong uniqueness} property is known for measure-valued solutions, \cite{Brenier11}: for $d \in \{2, 3\}$, if the initial datum $\bu_0$ is such that a classical solution $\bu$ exists for all time, then there exists a unique measure-valued solution with $\nu_{t,x} = \delta_{\bu(t,x)}$. That is, the measure-valued solution takes, at point $(t,x)$ the value of $u(t,x)$ with probability one; this is termed an \emph{atomic measure}.

Measure-valued solutions have, however, poor properties of uniqueness. Even in the simpler case of one-dimensional conservation laws (cf. Example 9.1 in \cite{FMTActa}), it is easy to construct different measure-valued solutions for the same stochastic initial datum. Young measures are pointwise-parameterized probability measures, which can be thought of as marginals of the joint probability measure of the values of the function globally. It is clear that, without additional information about correlations, marginals alone cannot uniquely describe a joint measure.

As a remedy to this lack of uniqueness, the framework of \emph{statistical solutions} has been proposed. Within it, a solution is sought in the space of time-parameterized probability distributions in $L^p$. There is ample literature on the topic, e.g. \cite{Foias1973,Vishik1979,Foias2013}. In this work, we will follow the ideas and notation of \cite{FLMfoundations}, in which the authors proved that time-parameterized probability distributions in $L^p$ can be uniquely identified with a \emph{correlation measure}, an infinite hierarchy of Young measures corresponding to the marginal distributions for finite tuples of points. We will define this rigorously in the sequel.

There exist in the literature efficient algorithms for the computation of statistical solutions for hyperbolic conservation laws, both scalar and systems thereof, e.g. \cite{FMTActa,FLMWSystems,LyePhd}; as well as for the incompressible Euler equations with spectral hyper-viscosity schemes, \cite{LMP1}. However, to our knowledge, no such method is available for the incompressible Euler equations outside of spectral schemes; this poses strict limitations on the set of spatial domains numerical schemes can be applied to. Hence, the goals of this paper are (a) present the fundamental concepts about statistical solutions of the incompressible Euler equations; (b) derive a numerical scheme for the practical approximation of the same, based on finite volumes; and (c) rigorously prove that the approximations produced by this scheme, under experimentally verifiable assumptions, converge to a statistical solution.

The research presented here is based on the author's work in \cite{CPPThesis}.

\section{Statistical solutions of the incompressible Euler equations}
\subsection{Correlation measures}
Let us employ the following notation: for a finite, signed Radon measure $\mu \in \mathcal{M}(\R^d)$, and a continuous function $f \in C(\R^d)$, we denote
\[ \langle \mu, g(\xi) \rangle \coloneqq \int_{\R^d} g(\xi)\, d\mu(\xi).  \]
Throughout, we implicitly assume that $\xi \in \R^d$ is the argument of $g$, in order to write expressions like $\langle \mu, \xi \otimes \xi \rangle$ (rather than $\langle \mu, \text{id} \otimes \text{id} \rangle$). For a metrized set $\Omega$ with the implicit $\sigma$-algebra of Borel's sets, we denote by $\Prob(\Omega)$ the set of all probability distributions defined on $\Omega$.

\begin{definition}\label{def:Youngmeasure}
	A \emph{Young measure} is a map $\nu \colon \domain \to \Prob(\R^{k})$ which is measurable in the weak* topology; i.e., for any function $g \in C_0(\R^{k})$, the map 
	\[x \mapsto \langle \nu_x, g \rangle \]
	is Borel-measurable. We denote for convenience $\nu_x \coloneqq \nu(x)$.
	
\end{definition}

In the definitions below, we denote by $B_r(x) \subset \domain$ the ball of center $x$ and radius $r$ in Euclidean norm; and by $\fint$ the averaged integral, $\fint_Af(x) \,dx = \frac{1}{\L(A)}\int_A f(x)\,dx$, with $\L$ the Lebesgue measure.

\begin{definition}\label{def:corrmeas}
	(\cite{FLMfoundations}) 
	\begin{sloppypar}
	A \emph{correlation measure} is a collection ${\bm \nu} = (\nu^1, \nu^2, \dots)$ of maps, with ${\nu^k \colon \domain^k \to \Prob\left(\Uk\right)}$ satisfying the following properties:
	\end{sloppypar}
	\begin{enumerate}
		\item\textit{Weak* measurability:} For all $k \in \N$, the map $\nu^k \colon \domain^k \to \Prob\left(\Uk\right)$ is weak* measurable; i.e. $\nu^k$ is a \emph{Young measure} from $\domain^k$ to $\Uk$.
		\item\textit{$L^p$-boundedness:} there exists $p \in [1, \infty)$ such that $\nu^1$ is $L^p$-bounded, in the sense that
		\begin{equation*}
		\int_{\domain} \langle \nu^1_x , \|\xi\|^p\rangle\,dx < \infty.
		\end{equation*}
		\item\textit{Symmetry:} If $\sigma$ is a permutation of $\{1,\dots,k\}$ and $f\in C_0(\R^k)$ then $\langle \nu^k_{\sigma(\bx)}, f(\sigma(\bm{\xi}))\rangle = \langle \nu^k_{\bx} , f(\bm{\xi}) \rangle$ for a.e.\ $\bx\in \domain^k$. Here, we denote $\sigma(\bx) = \sigma(x_1,x_2,\ldots, x_k) = (x_{\sigma_1},x_{\sigma_2},\ldots,x_{\sigma_k})$.
		\item\textit{Consistency:} If $f\in C_0\left(\Uk\right)$ is of the form $f(\xi_1,\dots,\xi_k) = g(\xi_1,\dots,\xi_{k-1})$ for some $g\in C_0(\Ukmone)$, then $\langle \nu^k_{x_1,\dots,x_k} , f \rangle = \langle \nu^{k-1}_{x_1,\dots,x_{k-1}}, g\rangle$ for almost every $(x_1,\dots,x_k)\in \domain^k$.
		\item\textit{Diagonal continuity:} 
		\begin{equation*}
		\lim_{r\to0}\int_\domain\fint_{B_r(x)}\langle \nu^2_{x,y}, \|\xi_1-\xi_2\|^2\rangle\,dy\,dx = 0.
		\end{equation*}
	\end{enumerate}
	Each element $\nu^k$ is called a \emph{correlation marginal}. We let $\L^p(\domain;\U)$ denote the \emph{set of all correlation measures} from $\domain$ to $\U$,
	\[ \L^p(\domain;\U) \coloneqq \left\{{\bm{\nu}} = (\nu^1, \nu^2, \dots),~\nu^k \colon \domain^k \to \Prob\left(\Uk\right), \bm{\nu} \text{ is a corr. meas., } \int_D \langle \nu^1_x , \|\xi\|^p\rangle\,dx < \infty \right\} \]
\end{definition}

It is known, \cite{FLMfoundations}, that correlation measures and probability measures in $L^p$ can be uniquely identified, through the following \emph{main theorem of correlation measures}:
\begin{theorem}
	\label{thm:duality}
	(\cite{FLMfoundations}, Theorem 2.7) 
	\begin{sloppypar}
		For every correlation measure $\bnu \in \L^p(\domain;\U)$, there exists a unique probability measure ${\mu \in \Prob(L^p(\domain, \U))}$ satisfying
		\begin{equation}
		\label{eq:finitemoment}\int_{L^p(\domain;\U)} \|\bu\|^p_{L^p(\domain;\U)} \,d\mu(\bu) < \infty,
		\end{equation}
		such that $\forall k \in \N$, and $\forall g \in L^1\left(\domain^k; C_0\left(\Uk\right)\right)$,
	\end{sloppypar}
	\begin{equation}
	\label{eq:duality}
	\int_{\domain^k} \int_{(\U)^k} g(\bx, \bm{\xi}) d\nu_x^k(\bm{\xi}) \,d\bx = \int_{L^p(\domain, \U)} \int_{\domain^k} g(\bx, \bu(\bx)) \,d\bx d\mu(\bu), 
	\end{equation}
	where $\bu(\bx)$ denotes the vector $(\bu(x_1), \bu(x_2), \dots, \bu(x_k))$, and $g(\bx, \bm{\xi}) \coloneqq g(\bx)(\bm{\xi})$. Conversely, for every probability measure $\mu \in \Prob(L^p(\domain,\U))$ with finite moment \eqref{eq:finitemoment}, there exists a unique correlation measure $\bnu \in \L^p(\domain; \U)$ satisfying \eqref{eq:duality}.
	
	Moreover, the (tensor-valued) moments 
	\[ m^k(\bx) \coloneqq \langle \nu_\bx^k, \,\xi_1 \otimes \xi_2 \otimes \dots \otimes \xi_k \rangle \]
	uniquely determine the correlation measure $\bnu$, and thus its associated $\mu$.
\end{theorem}

\subsection{Time-parameterized probability measures in function spaces}
In this section we introduce the mathematical space in which we will seek solutions of the incompressible Euler equations: \emph{time-parameterized correlation measures}; or equivalently, by Theorem \ref{thm:duality}, time-parameterized probability measures in $L^2$.

Fix $T \in [0, \infty)$. We denote $L^p_x \coloneqq L^p(\domain; \mathbb{R}^d)$; and $L^q_t(\mathcal{X}) \coloneqq L^q([0, T], \mathcal{X})$, for $\mathcal{X}$ a Banach space.

\begin{definition} (\cite{LMP1})
	\begin{enumerate}
		\item Consider a time-parameterized probability measure $\mu\colon [0, T) \to \Prob(L^2_x)$, and denote $\mu_t \coloneqq \mu(t)$. If for all $F \in C_b(L^2_x)$ (i.e. continuous and bounded), the mapping \[t \mapsto \int_{L^2_x} F(\bu) \, d\mu_t(\bu)\] is measurable for a.e. $t \in [0, T)$, we say that $\mu$ is \emph{weak* measurable}.
		
		\item We denote by $L^1_t(\Prob)$ the space of weak* measurable, time-parameterized probability measures such that
		\[ \int_0^T \int_{L^2_x} \|\bu\|_{L^2_x}\,d\mu_t(\bu) dt < \infty. \]
		
		\item\label{subdef:timereg} A weak* measurable, time-parameterized probability measure $\mu$ is called \emph{time-regular} if there exist a constant $L\in \N$ and a mapping $(s,t) \mapsto \pi_{s,t} \in \Prob(L^2_x \times L^2_x)$ such that for a.e. $s,t \in [0, T)$: \begin{itemize}
			\item The measure $\pi_{s,t}$ is a transport plan from $\mu_s$ to $\mu_t$. That is: for all measurable $A\subset L^2_x$, it holds that $\pi_{s,t}(A \times L^2_x) = \mu_s(A)$ and $\pi_{s,t}(L^2_x \times A) = \mu_t(A)$.
			\item There exists a constant $C>0$ such that
			\begin{equation}
			\label{eq:timeregulartransport} \int_{L^2_x \times L^2_x} \|\bu-\bv\|_{H_x^{-L}} d\pi_{s,t}(\bu,\bv) \le C|t-s|.
			\end{equation}
		\end{itemize}
		
		\item A family of time-parameterized, time-regular probability measures $\{\mu^\Delta_t\}_{\Delta > 0}$ is \emph{uniformly time-regular} if the constants $L$ and $C$ in point \ref{subdef:timereg} can be chosen independently of $\Delta$.
	\end{enumerate}
\end{definition}

We conclude this subsection with the following two results, which will be of use in the sequel.

\begin{lemma}(\cite{LMP1}, Proposition 2.2) 
	\label{lemma:limittimecontinuous}
	Let $\mu_t^\Delta \in L^1_t(\Prob)$ be a family of uniformly time-regular probability measures, for $\Delta > 0$. Assume there exists $R>0$ such that $\mu_t^\Delta(B_R(0)) = 1$ for all $\Delta > 0$ and a.e. $t \in [0, T)$, and $B_R(0)$ the ball of radius $R$ and center $0$ in $L^2$. If there exists $\mu_t \in L^1_t(\Prob)$ such that $\mu^\Delta_t \to \mu_t$, i.e.,
	\[ \lim_{\Delta \to 0^+} \int_0^T W_1(\mu^\Delta_t, \mu_t) \,dt = 0,\]
	then $\mu_t$ is time-regular, with the same time-regularity constants $C, L$ as the family $\mu_t^\Delta$.
\end{lemma}

We recall that for a Banach space $X$, and $\rho, \sigma \in \Prob(X)$, the $1$-Wasserstein metric $W_1$ is defined as
\[ 	\label{eq:wassersteindef} W_1(\rho, \sigma) \coloneqq \inf \int_{X \times X} \|\xi - \zeta\| \,d\pi(\xi, \zeta), \]
where the infimum is taken over all transport plans $\pi$ between $\rho$ and $\sigma$.

\begin{theorem}
	\label{thm:compactness}
	(\cite{LMP1}, Theorem 2.4) 
		Let $\{\mu^h_t\}_{h > 0}$ be a family of uniformly time-regular, weak* measurable probability measures in ${L^1_t(\Prob)}$, and assume that there exists $R>0$ such that 
	\begin{equation*}
	\mu_t^h(B_R(0)) = 1,\qquad \forall h > 0, \forall t \in [0, T).
	\end{equation*}
	
	Let $\bnu^h_t = (\nu^{h, 1}_t, \nu^{h, 2}_t, \dots)$ denote the corresponding time-parameterized correlation measures. If there exists a map $\Upsilon \colon [0, \infty) \to [0, \infty)$ with $\lim\limits_{r\to0^+}\Upsilon(r) = 0$ such that for all $h > 0$,
	\begin{equation*}
	 S_r^2(\bnu^h, T)^2 \coloneqq \int_0^T \int_D\fint_{B_r(x)}\langle \nu^{h,2}_{t; x,y}, \|\xi_1-\xi_2\|^2\rangle\,dy\,dx\,dt \le \Upsilon(r),
	\end{equation*}	
	then there exist a subsequence $h_j \to 0$, $j \in \N$ and a time-parameterized probability measure $\mu_t \in L^1_t(\Prob)$ such that
	\begin{equation}
	\label{eq:convW1} \int_0^T W_1(\mu_t^{h_j}, \mu_t)\,dt \stackrel{j\to\infty}{\longrightarrow} 0.
	\end{equation}
	
	Furthermore, if we denote $\bnu_t = (\nu^{1}_t, \nu^{2}_t, \dots)$ the time-parameterized correlation measure associated to the limit $\mu$, the following properties are preserved:
	\begin{enumerate}[I.]
		\item $L^2$-bound: for a.e. $t \in [0, T)$, 
		\[\int_\domain \langle \nu^1_{t, x}, \|\xi\|^2 \rangle \,dx \le R^2. \]
		\item The two-point correlations satisfy
		\[ S_r^2(\bnu, T)^2 \le \Upsilon(r). \]
		\item \label{prop:strongconv_time} We say that a function $g \in C([0, T)\times \domain^k \times \Uk)$, is an \emph{admissible observable} if there exists $C>0$ such that for all $t \in [0, T)$, $\bx \in \domain^k$, $\bm{\xi}, \bm{\xi}' \in \Uk$,
		\begin{equation*}
		\begin{split} |g(t, \bx, \bm{\xi})| &\le C \prod_{i=1}^k \left( 1 + \|\xi_i\|^2 \right), \\
		|g(t, \bx, \bm{\xi}) - g(t, \bx, \bm{\xi}')| &\le C \sum_{i=1}^k L_i(\bm{\xi}, \bm{\xi}') \|\xi_i - \xi_i'\| \sqrt{1 + \|\xi_i\|^2 + \|\xi'_i\|^2},
		\end{split}
		\end{equation*}
		with 
		\[ L_i(\bm{\xi}, \bm{\xi}') \coloneqq \prod_{j = 1, j \neq i}^k \left( 1 + \|\xi_i\|^2 + \|\xi'_i\|^2 \right). \]
		
		Then, any admissible observable $g$ converges strongly in $L^1([0, T) \times \domain^k)$ in expectation, i.e.,
		\[ \lim_{j \to \infty} \int_0^T \int_{\domain^k} \left|\langle \nu_{t,\bx}^{h_j, k},\, g(t, \bx, \xi)\rangle - \langle \nu_{t,\bx}^{k},\, g(t,\bx, \xi)\rangle \right| \,d\bx dt = 0.\]
	\end{enumerate}
\end{theorem}

\subsection{Statistical solutions of the incompressible Euler equations}
With the tools presented before, we can finally rigorously define the object of interest of this work: \emph{statistical solutions} of the incompressible Euler equations. We follow notation here from \cite{FLMfoundations,LMP1}.
\begin{definition}\label{def:statsol}
	(\cite{LMP1}, Def. 3.1) A time-parameterized probability measure $\mu_t \in L^1_t(\Prob)$ is a \textbf{statistical solution of the incompressible Euler equations} with initial data $\bar{\mu} \in \Prob(L^2(\domain; \R^d))$ if $t \mapsto \mu_t$ is time-regular, and the associated correlation measure $\bm{\nu}_t$ (in the sense of Theorem \ref{thm:duality}) satisfies:
	\begin{enumerate}
		\item Given ${\phi}_1, \dots, {\phi}_k \in C^\infty([0,T)\times \domain;\mathbb{R}^d)$ with $\div({\phi}_i) = 0$ for all $i=1,\dots,k$, set
		\[
		{\phi}(t,\bx)
		= 
		{\phi}_1(t,{x}_1) \otimes \dots \otimes {\phi}_k(t,{x}_k),
		\quad
		\text{where }\bx=({x}_1,\dots,{x}_k) \in \domain^k.
		\]
		Then $\nu^k = \nu^{k}_{t,{x}_1,\dots,{x}_k}$ satisfies
		\begin{gather} 
		\label{eq:statsol1}
		\begin{aligned}
		\int_0^T \int_{\domain^k} &\langle \nu^{k}, {\xi}_1 \otimes \dots \otimes {\xi}_k \rangle \colon \partial_t {\phi}
		+ \sum_{i=1}^k \langle \nu^{k}, {\xi}_1 \otimes \dots \otimes {F}({\xi}_i) \otimes \dots \otimes  {\xi}_k \rangle \colon \nabla_{{x}_i} {\phi}
		\, d\bx \, dt
		\\
		&\quad + \int_{\domain^k} \langle \bar{\nu}^{k}, {\xi}_1 \otimes \dots \otimes {\xi}_k \rangle \colon {\phi}(0,\bx) \, d\bx
		=
		0.
		\end{aligned}
		\end{gather}
		Here $\bar{\nu}$ is the correlation measure corresponding to the initial data $\bar{\mu}$. We denote ${F}({\xi}) \coloneqq {\xi}\otimes {\xi}$ and the contraction in the second term is more explicitly given by
		\[
		\left({\xi}_1 \otimes \dots \otimes {F}({\xi}_i) \otimes \dots \otimes  {\xi}_k \right)\colon \nabla_{{x}_i} {\phi}
		=
		\left[
		\textstyle\prod_{j\ne i} \left({\xi}_j \cdot {\phi}_j\right)
		\right] 
		\left(
		{\xi}_i\cdot \nabla_{{x}_i} {\phi}_i 
		\right)\cdot {\xi}_i. 
		\]
		\item For all $\psi \in C_c^\infty(\domain)$, and for a.e. $ \in [0, T)$,
		\begin{gather}
		\label{eq:statsol2}
		\int_{\domain^2} \langle \nu^2_{t,x,y}, {\xi}_1 \otimes {\xi}_2 \rangle \colon \left(\nabla \psi(x) \otimes \nabla \psi(y)\right)
		\, dx \, dy = 0.
		\end{gather}
	\end{enumerate}
\end{definition}

As is the case for weak solutions (cf. Lions' notion of dissipative solutions, \cite{Lions96}), a stronger definition of solution, including a condition on energy dissipation, provides better theoretical properties. This is the object of the following:

\begin{definition}
	\label{def:dissipativestatsol}
	(\cite{LMP1}, Def. 3.2) A statistical solution $\mu_t \in L^1_t(\Prob)$ of the incompressible Euler equations with initial data $\bar{\mu} \in \Prob(L^2_x)$ is called a \emph{dissipative statistical solution} if, for all $M \in \N$, for every choice of coefficients $\bm{\alpha} = (\alpha_i)_{i=1}^M \in (0, 1]$ with $\sum_{i=1}^M \alpha_i = 1$, and for every $(\bar{\mu}_1, \dots, \bar{\mu}_M) \in \Prob(L^2_x)^M$ with $\sum_{i=1}^M \alpha_i  \bar{\mu}_i = \bar{\mu}$, there exists a function $t \mapsto (\mu_{1,t}, \dots, \mu_{M,t})$ with $\sum_{i=1}^M \alpha_i \mu_{i} = \mu$ such that, for all $i \in \{1, \dots, M\}$:
	\begin{itemize}
		\item $t \mapsto \mu_{i,t}$ is weak* measurable, with $\mu_{i,t} |_{t=0} = \bar{\mu}_i$;
		\item for all $\phi \in C_c^\infty([0, T) \times \domain)$ with $\div(\phi) \equiv 0$,
		\begin{align*}
		\begin{split}
		\int_0^T \int_{L_x^2} \int_\domain &\left[ \bu \cdot \partial_t \phi + (\bu \otimes \bu) \colon \nabla \phi \right]\,dx\,d\mu_{i,t}(\bu)\,dt
		= - \int_{L^2_x} \int_\domain \bu \cdot \phi(0, x)\,dx \,d\bar{\mu}_i(\bu);
		\end{split}
		\end{align*} 
		\item for a.e. $t \in [0, T)$, and for $i \in \{1, \dots, M\}$,
		\begin{equation*}
		\int_{L^2_x} \|\bu\|_{L^2_x}^2 \,d\mu_{i,t}(\bu) \le \int_{L^2_x} \|\bu\|_{L^2_x}^2 \, d\bar{\mu}_i(\bu) .
		\end{equation*}
	\end{itemize}
\end{definition}

Dissipative statistical solutions are a recently derived framework, for which the currently available results appear promising. Specifically, the two following results are proven in \cite{LMP1}. The first is that, under a regularity condition, dissipative statistical solutions satisfy short-time existence and uniqueness:

\begin{theorem} (\cite{LMP1}, Corollary 3.1)
	\label{thm:localeu}
	If $s \ge \lfloor d/2 \rfloor + 2$, and if $\exists C >0$ such that the initial data $\bar{\mu} \in \Prob(L^2_x)$ is concentrated on
	\[ \{ \bar{u} \in H^s(\domain; \R^d) \colon \|\bar{u} \|_{H^s(\domain; \R^d)} \le C \},\]
	then there exists $T^*>0$, depending only on $C$, and a statistical solution $\mu_t \colon [0, T^*] \to \Prob(L^2_x)$ with initial data $\bar{\mu}$.	
	Furthermore, $\mu_t$ is unique in the class of dissipative statistical solutions for $t \in [0, T^*]$.
\end{theorem}

The second result is that restricting to two dimensions, and under smoothness conditions, dissipative statistical solutions verify global existence and weak-strong uniqueness in the following sense:
\begin{theorem}(\cite{LMP1}, Corollary 3.2)
	\label{thm:wsu}
	Let $d=2$, and let $\alpha \in (0,1)$. If $\bar{\mu}$ is concentrated on $C^{1,\alpha}(\domain; \R^d)$, and if there exists $M>0$ such that $\bar{\mu}(B_M(0)) = 1$, then there exists a dissipative statistical solution $\mu_t$ with initial data $\bar{\mu}$. Furthermore, $\mu_t$ is unique in the class of dissipative statistical solutions with initial data $\bar{\mu}$.
\end{theorem}

\section{A deterministic finite volume scheme for the incompressible Euler equations}
\label{sec:scheme}
In this section we will present a numerical scheme for the incompressible Euler equations, originally presented in \cite{Leonardi18}, and based on the work of \cite{BCG89}. This is an efficient finite volume method based on a discrete Leray projection.

Let $\bem$ be the $m$-th unit vector in the canonical basis of $\mathbb{R}^d$, $(\bem)_i = \delta_{i,m}$. We assume a partition of $D$ into a uniform Cartesian grid, with
\[ \Dx_1,~ \Dx_2, \, \dots,~  \Dx_d = O(h).\]
Let the set of cells be $\{ C_{\bi} \}_{\bi \in I}$, with $C_{\bi} \coloneqq \prod_{m=1}^d [i_m \Dx_m, (i_m+1)\Dx_m]$, for $\bi = (i_1, \dots, i_d) \in I$ the set of indices in the Cartesian mesh, $I = \{1, \dots, N_1\} \times \dots \times \{1, \dots, N_d\}$. We refer by $x_\bi$ to the midpoint of cell $C_\bi$. We term $\Gh$ the set of piecewise constant functions in the grid, i.e.,
\begin{equation*}
\Gh^k \coloneqq \{ \bu \colon D \to \mathbb{R}^k \colon \bu\big|_{\mathring{C_{\bi}}} \text{ constant }\forall \bi \in I \}.
\end{equation*}

We implicitly identify a function in $\Gh^k$ with the vector of cell averages in $\mathbb{R}^{k|I|}$; e.g. for $\bu \in \Gh^k$, we write $ \| \bu\|_{L^2(D)}^2  = \sum_{\bi \in I} \| \bu_\bi \|_2^2 \Dx_1 \cdots \Dx_d$. 

For $\bu = (u_1, \dots, u_d) \in \Gh^d$ (i.e. $u_i \in \Gh^1$ for all $i$), and $\varphi \in \Gh^1$, we define the following operators:
\begin{align*}
	\divh \colon \Gh^d \to \Gh^1 ; \qquad &(\divh \bu)_{\bi} \coloneqq \sum_{m=1}^d \dfrac{(u_m)_{\bi + \bem } - (u_m)_{\bi-\bem} }{2 \Dx_m}, \\
	\gradh \colon \Gh^1 \to \Gh^d ; \qquad &(\gradh \varphi)_{\bi} \coloneqq \left[
		\dfrac{\varphi_{\bi + \be_1} - \varphi_{\bi - \be_1}}{2\Dx_1}, \dots, 
		\dfrac{\varphi_{\bi + \be_d} - \varphi_{\bi - \be_d}}{2\Dx_d} \right], \\
	\laplh \colon \Gh^1 \to \Gh^1 ; \qquad & (\laplh \varphi)_{\bi} \coloneqq \sum_{m=1}^d \frac{\varphi_{\bi + 2\bem} - 2\varphi_{\bi} + \varphi_{\bi - 2\bem} }{4 \Dx_m^2}.
\end{align*}

We will denote by $\Ghdiv$ the vector subspace in $\Gh^d$ of discretely divergence-free functions, i.e.
\[ \Ghdiv = \{ \bu \in \Gh^d \colon \divh \bu \equiv 0 \}. \]

Recall that we will assume for convenience periodic boundaries. A first trivial observation:
\begin{lemma}
	\label{lemma:adj}
	Let $\bu \in \Gh^d$, $\psi \in \Gh^1$. Then it holds that:
	\begin{equation*}
	\sum_{\bi \in I} \bu_\bi \cdot \gradh \psi_\bi = - \sum_{\bi\in I} \psi_\bi ~\divh \bu_\bi.
	\end{equation*}
\end{lemma}

With this notation, we are finally in position of defining the finite volume scheme of \cite{Leonardi18}. Fix $\theta \in \big(\frac{1}{2}, 1\big]$. The scheme has the following form; a detailed explanation of each operator therein follows.
\begin{align}
\label{eq:scheme1}
\frac{{\bus}^{,n+1} - \bu^n }{\Dt^n} + \bm{C}(\bu^n, \bbu) &= \bm{D}(\bbu)\\
\frac{\bu^{n+1} - \bu^n}{\Dt^n} &= \bP\left( \frac{{\bus}^{,n+1} - \bu^n}{\Dt^n} \right) \label{eq:scheme2},
\end{align}
coupled with the initial condition $\bu^0 = \Phdiv(\barbu) \in \Ghdiv$, and with time-averaged velocity $\but^{n+\frac{1}{2}} \in \Gh^d$ defined as
\begin{equation*}
\bbu := \theta {\bus}^{,n+1} + (1-\theta) \bu^n.
\end{equation*}
In the scheme above, the operators are defined as follows:
\begin{itemize}
	\item Operator $\bC$ approximates the convective term $[(\bu \cdot \nabla)\bu]$ as
	\begin{align*}
	\begin{split}
	\bm{C}&\colon \Gh^d \times \Gh^d \to \Gh^d \\
	\bm{C}(\bu, \bv)_\bi &= \sum_{m=1}^d \dfrac{\bFm(\bu_{\bi}, \bu_{\bi + \bem}, \bv_{\bi}, \bv_{\bi + \bem}) - \bFm(\bu_{\bi-\bem}, \bu_{\bi}, \bv_{\bi-\bem}, \bv_{\bi})}{\Dx_m} 
	\end{split}
	\end{align*}
	with, denoting $\bu^L \coloneqq (u^L_i)_{i=1}^d$, and analogously for $\bu^R$:
	\begin{align*}
	\begin{split}
	&\bFm \colon \Gh^d \times \Gh^d \times \Gh^d \times \Gh^d \longrightarrow \Gh^d\\
	&\bFm(\bu^L, \bu^R, \bv^L, \bv^R) = 
	\frac{u_m^L + u_m^R}{4} ( \bv^L + \bv^R).
	\end{split}
	\end{align*}
	
	\item Operator $\bD$ is a Lax-Wendroff second order diffusion operator,
	\begin{equation*} 
	\begin{split}
	\bD&\colon \Gh^d \to \Gh^d \\
	\bm{D}(\bu)_\bi &= \epsilon \sum_{m=1}^d \frac{\|\jump{\bu}_{\bi, m}\|_2 \jump{\bu}_{\bi, m} - \|\jump{\bu}_{\bi - \bem, m}\|_2 \jump{\bu}_{\bi - \bem, m}}{\Dx_m}, 
	\end{split}
	\end{equation*}
	with $\epsilon>0$ fixed, and $ \jump{\bu}_{\bi, m} \coloneqq \bu_{\bi + \bem} - \bu_{\bi}.$
	
	\item Operator $\bP$ is a discretely divergence-free projection, given by
	\begin{equation*}
	\begin{split}
	\bP &\colon \Gh^d \to \Ghdiv\\
	\bP(\bu) &= \bu - \gradh \psi,
	\end{split}
	\end{equation*} 
	where $\psi (\equiv \psi(\bu)) \in \Gh^1$ is the solution of the linear system of equations
	\begin{equation*} \laplh \psi = \divh \bu, \end{equation*}
	paired with suitable boundary conditions. It is not difficult to see, cf. Lemma 3.22, \cite{CPPThesis}, that this is indeed a well-defined, orthogonal projection operator, which acts as a linear transformation of $\bu$.
	
	\item For $h>0$, and for $p \ge 1$, let $A^h \colon  L^p(D)^d \to \Gh^d$ be the \emph{cell averaging operator}, given by $\left(A^h(\phi)\right)_\bi \coloneqq \fint_{C_\bi} \phi(\bx)\,d\bx$. Let $\Phdiv\colon L^2(D; \R^d) \to \Ghdiv$ be defined by
	\[ \Phdiv \equiv \bP \circ A^h.\]
\end{itemize}

We now list some elementary properties of scheme \eqref{eq:scheme1}-\eqref{eq:scheme2}. Proofs are omitted for brevity, but they are straightforward with standard techniques; they can all be found in \cite{CPPThesis}.

\begin{lemma}  \label{lemma:schemeprops} Assume there exists $\lambda>0$, independent of $h$, with $\Dt^n < \lambda h$ $\forall n$. Then
	\begin{enumerate}[I.]
		\item Projection operator $\bP$ is energy-dissipative, i.e. for all $\bus \in \Gh^d$, 
		\begin{equation*}
		\|\bP (\bus_\bi)\|_{L^2} \le \|\bus_\bi\|_{L^2}.
		\end{equation*}
		
		\item \label{lemma:schemeprops:L2} Scheme \eqref{eq:scheme1}-\eqref{eq:scheme2} is \emph{energy-stable} or $L^2$-stable; i.e.,
		\begin{equation*}
		\| \bu_\bi^{n+1} \|_{L^2} \le \| \bu_\bi^n \|_{L^2}.
		\end{equation*}
		
		\item For all $m \in \{1, 2, \dots, d\}$, scheme \eqref{eq:scheme1}-\eqref{eq:scheme2} satisfies:
		\begin{align*}
		\sum_{n=0}^{N-1} \sum_{\bi \in I} \| \bbu_{\bi + \bem} - \bbu_{\bi} \|_2^3 \Dx_1 \cdots \Dx_d \Dt^n &\le C \| \bu^0 \|_{L^2(D)}^2 h, \\
		\sum_{n=0}^{N-1} \sum_{\bi \in I} \| {\bus}_{\bi}^{,n+1} - \bu_{\bi}^n \|_2^2 \Dx_1 \cdots \Dx_d \Dt^n &\le C \| \bu^0 \|_{L^2(D)}^2 h, 
		\end{align*}
		with $C$ independent of $\bu$ and $h$.
		
		\item \label{lemma:schemeprops:tv} Scheme \eqref{eq:scheme1}-\eqref{eq:scheme2} satisfies the following \emph{total variation-type property}:
		\begin{equation*} \sum_{m=1}^d \sum_{n=0}^{N-1} \sum_{\bi \in I} \left\| \bu_{\bi + \bem}^n - \bu_{\bi}^n \right\|_2^\tvxl \Dx_1 \dots \Dx_d \Dt^n \le C \|\bu^0\|_{L^2(D)}^{\frac{4}{3}} h^\frac{2}{3},  \end{equation*}
		where $C$ is independent of $\bu$ and $h$.
		
		\item \label{lemma:schemeprops:D} Operator $\bD$ (weakly) vanishes in the limit. That is: for all $\phi \in C^\infty([0, T) \times D; \R^d)$ such that for a.e. $t \in [0, T)$, $\div(t, \cdot) \equiv 0$, let $\phi^h(t, \cdot) \coloneqq \Phdiv \circ \phi(t, \cdot)$. Then there exists $C>0$ independent of $\phi$, $\bu$ and $h$ such that
		\[ \left| \int_0^T \int_D \bD(\but(t, x)) \cdot \phi^h(t, x) \,dx \,dt \right| \le C  \|\bu_0\|_{L^2(D)}^\frac{4}{3} h^\frac{2}{3}. \]
	\end{enumerate}
\end{lemma}

Additionally the following properties are useful for the main results of this work in the sequel. Their proofs are fairly technical; we refer the interested reader to Appendix A in \cite{CPPThesis}.
\begin{lemma} \label{lemma:Phdivprops}
	Projection operator $\Phdiv \colon L^2(D; \R^d) \to \Ghdiv$ satisfies the following properties:
	\begin{enumerate}[I.]
		\item \label{lemma:Phdivprops:smalldiv} Let $\phi \in C^2(D; \R^d)$ with $\div(\phi) \equiv 0$. There exists $C = C\left(\max\limits_{|\alpha|=2} \left|D^\alpha \phi\right|\right)>0$ such that
		\begin{equation*}
		\| \phi - \Phdiv \circ \phi \|_{L^2(D; \R^d)} \le C h.
		\end{equation*}
		
		\item \label{lemma:Phdivprops:L2bd} There exists $C(d)$ such that for all $\phi \in L^2(D; \R^d)$, and for all $h>0$,
		\begin{equation*}  \| \Phdiv (\phi) \|_{L^2(D; \R^d)} \le C(d) \| \phi \|_{L^2(D; \R^d)}.
		\end{equation*}
		\item \label{lemma:Phdivprops:Linfbd} If $\phi \in C^2(D; \R^d)$ has $\div(\phi) \equiv 0$, and $\phi^h \coloneqq \Phdiv(\phi)$, then
		\begin{align*}
		\begin{split}
		\lim_{h\to 0^+} \max_{\bi \in I}\| \phi^h_\bi - \phi(x_\bi) \|_{\infty} = 0, \\
		\lim_{h \to 0^+} \max_{\bi \in I} \left\| \frac{\phi_{\bi+\bem}^h - \phi_\bi^h}{h} - \partial_{x_m} \phi(x_\bi) \right\|_\infty = 0.
		\end{split}
		\end{align*}	
		
		\item $\Phdiv$ preserves time-regularity. That is, if $\phi \in L^\infty([0, T] \times D, \R^d)$ is such that $\phi(\cdot, x) \in C^1([0, T], \R^d)$ for a.e. $x$, then the function $\tilde{\phi}(t, \cdot) \coloneqq \Phdiv \circ \phi(t, \cdot)$ also has $\tilde{\phi}(\cdot, x) \in C^1([0, T], \R^d)$ for all $x$.
	\end{enumerate}
\end{lemma}

\section{A convergent finite volume scheme for statistical solutions of the incompressible Euler equations}
In this section, we present the main results of this work. We introduce a Monte Carlo scheme, built upon the deterministic finite volume scheme presented in the previous section; and we rigorously show that it converges to a statistical solution. This convergence is only under an external assumption; in Section \ref{sec:examples} we will show numerical evidence that this hypothesis is mild for cases of practical interest.

The Monte Carlo scheme is based on the well-known concept of \emph{ensemble simulations}. Although it has long been present in the literature, especially for numerical weather forecasting, e.g. (\cite{Epstein69,Leith74}), we follow the more recent presentation in \cite{FKMT17} of an analogous method for hyperbolic conservation laws.

The specific formulation of the scheme that we follow is that of \cite{LeonardiPhd,Leonardi18}; we also employ the numerical solver of the same author, \cite{luqness}. In the original presentation, the scheme is considered as an approximation method for measure-valued solutions. In this work, we prove that the scheme contains, in fact, sufficient information to converge to a statistical solution.

Throughout this section, we consider the initial value problem for the incompressible Euler equations in the sense of \emph{uncertainty quantification}: we consider initial data randomly distributed according to a probability measure $\barmu \in L^2(D; \R^d)$. For a deterministic initial condition, we will adopt the same approach as e.g. \cite{FMTActa} and introduce an arbitrary, small random perturbation to the initial datum, and thus view it as a probability measure; we will discuss this in more detail in Section \ref{sec:examples}.

\subsection{Notation}
Let $\Lpdiv(D) \subset L^p(D; \R^d)$ be the space of weakly divergence-free functions in $L^p$, i.e.
\[ \Lpdiv(D; \R^d) := \left\{ \bu \in L^p(D; \R^d) : \int_D \bu \cdot \nabla F = 0,\ \forall F \in C^\infty(D; \R)\right\}. \]

For simplicity of the exposition, we will henceforth assume a uniform Cartesian mesh in all directions, i.e., for all $m \in \{ 1, \dots, d\}$ and $\bi \in I$, $ x_{\bi + \bem} - x_{\bi} =: h$. Let $\barbu \in \Ldiv(D)$. For $h>0$, let 
\begin{equation*}
\bu^{h; 0} := \Phdiv \circ \barbu \in \Ghdiv
\end{equation*}
and for all $n \in \{1, \dots, N\}$, let $\bu^{h; n}$ be obtained with algorithm \eqref{eq:scheme1}-\eqref{eq:scheme2} with $\bu^{h;0}$ as initial datum.
Thus we can define the following piecewise linear in time evolution operator:
\begin{align}
\begin{split}
\label{eq:Sdef}
&\mathcal{S}^h : [0, T] \times \Ldiv(D) \to \Ghdiv \cap L^2(D; \R^d)\\
&\mathcal{S}^h(t, \barbu)(\cdot) := \sum_{n=0}^{N(h)-1}\left(\frac{t^{n+1}-t}{\Dt^n}\bu^{h;n}(\cdot) + \frac{t - t^n}{\Dt^n}\bu^{h; n+1}(\cdot)\right) \indic_{[t^n, t^{n+1})}(t).
\end{split}
\end{align}
For brevity we write $\mathcal{S}^h_t \barbu  \coloneqq \mathcal{S}^h(t, \barbu)$. 
Given an initial distribution ${\bar{\mu} \in \Prob(\Ldiv(D; \R^d))}$, for all $t \in [0,T]$ we can define a probability measure in $L^2(D; \R^d)$ at time $t$, via push-forward measure:
\begin{equation}
\mu^h_t := \mathcal{S}^h_t \# \bar{\mu} \label{eq:discrevol}.
\end{equation}

\subsection{The Monte Carlo algorithm}
We can now present the numerical algorithm that we will use in this work, which follows the FKMT algorithm in \cite{LyePhd,FKMT17}.

\begin{algorithm}
	\label{algo:fkmt}
	\emph{(FKMT algorithm for the incompressible Euler equations)}.
	
	Let $\bar{\mu} \in \Prob(L^2(D ; \R^d))$; and for a mesh width $h>0$, let $\mathcal{S}^h$ be the evolution operator as in \eqref{eq:Sdef}. Fix a number of samples $M$.
	\begin{enumerate}
		\item Generate $\barbu_1, \barbu_2, \dots, \barbu_M \in L^2(D; \R^d)$ independent random variables with distribution $\bar{\mu}$.
		\item For $m \in \{1, \dots, M\}$, evolve the sample in time, $\bu^h_m(t) := \mathcal{S}^h_t(\barbu_{m})$.
		\item Estimate the statistical solution by the \emph{empirical measure}
		\begin{equation*}
		\mu_t^{h, M} := \frac{1}{M} \sum_{m=1}^M \delta_{\bu^{h}_m(\cdot, t)}.
		\end{equation*}
	\end{enumerate}
\end{algorithm}

\subsection{Convergence of the initial data}
\begin{lemma}
	\label{lemma:initconv}
	Let $\mu \in \Prob(L^2(D; \mathbb{R}^d))$ with $\bu \in \Ldiv$ $\mu$-almost surely. For $h>0$, define $\mu^h := \Phdiv \# \mu$.
	
	For all $G\in C^1(D\times \R^d; \R^k)$, 
	\begin{equation*}\lim_{h \to 0}
	\left\| \int_{L^2(D; \mathbb{R}^k)} \int_{D} G(x, \bu(x)) \,dx\,d\mu(\bu) - \int_{L^2(D; \mathbb{R}^k)} \int_D G(x, \bu(x)) \,dx \,d\mu^h(\bu)  \right\|_2 = 0.
	\end{equation*}
\end{lemma}
\begin{proof}
	$\mu^h$ is well-defined, as $\Phdiv$ is measurable by construction. Furthermore,
	\begin{align*}
	&\left\| \int_{L^2(D; \mathbb{R}^k)} \int_{D} G(x,\bu(x)) \,dx\,d\mu(\bu) - \int_{L^2(D; \mathbb{R}^k)} \int_{D} G(x, \bu(x)) \,dx \,d\mu^h(\bu) \right\|_2 \\
	&\le (\L(D))^\frac{1}{2} \| \nabla_\bu G\|_{L^\infty(D; \R^{k\times d})} \int_{L^2(D; \mathbb{R}^d)} \left\| \bu - \Phdiv\circ  \bu \right\|_{L^2(D; \R^d)}\,d\mu(\bu).
	\end{align*}
	
	Through Lemma \ref{lemma:Phdivprops}.\ref{lemma:Phdivprops:smalldiv}, $\mu$-a.s., $\left\| \bu - \Phdiv\circ  \bu \right\|_{L^2(D; \R^d)} \to 0$ as $h \to 0$, and the conclusion is immediate.
	
\end{proof}

\subsection{Time-regularity}
\begin{lemma}
	\label{lemma:uniformtimecont}
	Let $\bar{\mu} \in \Prob(L^2(D; \R^d))$ with $\bu \in \Ldiv$ $\bar{\mu}$-a.s, and such that there exists $R>0$ with $\supp(\bar{\mu}) \subset B_R(0)$, the ball of radius $R$ and center $0$ in $L^p(D; \mathbb{R}^d)$. Let $\{\mu^h\}_{h>0}$ be a family of time-parameterized probability measures generated by \eqref{eq:discrevol}, and assume $\Delta t^n = O(h)$. Then this family is uniformly time-regular.
\end{lemma}

\begin{proof} Fix $h>0$, and let $s,t \in [0, T]$ with $s < t$. Trivially a transport map between $\mu^h_t$ and $\mu^h_s$ exists.
	
	We now verify that there exist constants $C>0$, and $L \in \N$, independent of $h$, such that for any realization of the initial data $\bar{\mu}$, denoted by element $\omega$, 
	\[ \|\bu(\omega; t, \cdot)-\bu(\omega; s, \cdot)\|_{H_x^{-L}} \le C|t-s| \]
	which obviously implies eq. \eqref{eq:timeregulartransport}, $\int_{L^2_x \times L^2_x} \|\bm{w}-\bv\|_{H_x^{-L}} d\pi_{s,t}(\bm{w},\bv) \le C|t-s|$.
	
	We omit the parameter $\omega$ in the rest of this proof. Let $k\ge 2$ and let $\Phi \in H^k(D; \R^d)$. We will assume $\Phi$ is divergence-free in the sequel. If it is not, a straightforward argument of Helmholtz decomposition guarantees that the integral involving the gradient term is vanishingly small.
	For convenience, let us assume there exist $\tilde{n}$, $\tilde{m}$ such that $s = t^{\tilde{n}}$, $t = t^{\tilde{m}}$; the proof below only needs a minor, trivial modification otherwise. Then
	\begin{align*}
	&\left| \int_D \left( \bu(x, t) - \bu(x, s) \right) \cdot \Phi(x) \,dx \right| = \left| \sum_{n=\tilde{n}}^{\tilde{m}-1}  \int_D \int_{t^n}^{t^{n+1}} \partial_t \bu(x, \tau) \cdot \Phi(x) \,d\tau\,dx \right|.
	\end{align*}
	
	For any time interval $(t^n, t^{n+1})$, $\partial_t \bu$ is ($\bar{\mu}$-a.s.) well defined, and in fact has a constant value of $\frac{1}{\Delta t^n} (\bu^{n+1} - \bu^n)$. Therefore, by definition of the numerical scheme \eqref{eq:scheme1}-\eqref{eq:scheme2}, 
	
	\begin{align*}
	&\int_D \int_{t^n}^{t^{n+1}} \partial_t \bu(x, \tau) \cdot \Phi(x) \,d\tau = \Delta t^n \int_D \left[ -\bm{C}^h(\bu^n, \bbu) + \bm{D}^h(\bbu) - \gradh \psi^{h; n} \right] \cdot \Phi(x)\,dx
	\end{align*}
	
	Due to the density of smooth functions in $H^2$, we can assume $\Phi \in C^\infty$; otherwise the difference can be made arbitrarily small, and the $L^2$ bound in Lemma \ref{lemma:Phdivprops}.\ref{lemma:Phdivprops:Linfbd} allows one to disregard it. Furthermore, if $\Phi$ is smooth, the difference between its cell averages $\bar{\Phi}$ and $\Phi$ vanishes uniformly as $h \to 0$; and as $\Phi$ is divergence-free, \ref{lemma:Phdivprops}.\ref{lemma:Phdivprops:smalldiv} gives that $\|\Phi - \Phdiv \circ \Phi\|_{L^2}$ also vanishes uniformly.
	
	With an easy argument based on the Lipschitz-continuity of the fluxes $\bFm$ in $\bC$, $\int_D \bm{C}^h(\bu^n, \bbu) \cdot \bar{\Phi}(x) \,dx$ can be bounded, up to an additive error term that vanishes as $h \to 0$, by a term of the form
		\[ \sumbi h^d \sum_{m=1}^d (u^n_\bi)_m \bu_\bi^n \cdot \frac{\bar{\Phi}_{\bi+\bem} - \bar{\Phi}_{\bi}}{h} \]
		which in turn can be bounded by $\|\bu^n\|_{L^2}^2 \|\nabla \Phi\|_{L^\infty}$ by smoothness of $\Phi$.
		
		An analogous idea (cf. Lemma \ref{lemma:schemeprops}.\ref{lemma:schemeprops:D}) lets us transfer one discrete derivative from $\bD^h$ to $\bar{\Phi}$, obtaining a similar bound of the type $\|\bu^n\|_{L^2} \|\nabla \Phi\|_{L^\infty}$ for the term $\int_D \bm{D}^h(\bbu) \cdot \Phi(x) \,dx$. The term $\int_D \gradh \psi^{h; n}  \cdot \Phi^h(x) \,dx$ is identically zero through Lemma \ref{lemma:adj}.
	
	Now, choosing $2 \le L \in \N$ large enough that, by Sobolev's inequality, $H^L(D; \R^d)$ continuously embeds in $W^{1, \infty}(D; \R^d)$, we can bound the $L^\infty$ norms above by Sobolev norms. As $\|\bu\|_{L^2} \le R$, then, we can conclude that
	\begin{align*}
	\left| \int_D \left( \bu(x, t) - \bu(x, s) \right) \cdot \Phi(x) \,dx \right| \le C'(d) (t-s) R^2 \|\Phi\|_{H^L},
	\end{align*}
	and thus
	\[ \|\bu(\omega; t, \cdot)-\bu(\omega; s, \cdot)\|_{H_x^{-L}} \le R^2 C'(d)|t-s|. \]
	
\end{proof}

\begin{theorem}
	\label{thm:relativeconv}
	Fix $p \in [1, \infty)$. Let $\bar{\mu} \in \Prob(L^p(D; \mathbb{R}^d))$ represent a set of stochastic initial data for the incompressible Euler equations, with $\bu \in \Lpdiv$ $\bar{\mu}$-almost surely. Assume that there exists $R>0$ with $\supp(\bar{\mu}) \subset B_R(0)$, the ball of radius $R$ and center $0$ in $L^p(D; \mathbb{R}^d)$.
	For $h>0$, let $\mu^h_t = \mathcal{S}^h_t \# \mu$ as in eq. \eqref{eq:discrevol}. Assume furthermore that the discrete projection scheme \eqref{eq:scheme1}-\eqref{eq:scheme2} satisfies the following properties:
	\begin{itemize}
		\item $L^p$ bounds: $\exists C>0$ such that
		\begin{equation}
		\label{hyp:Lp}
		h^d \sum_{\bi \in I} \|\bu_\bi^{h; n} \|_2^p \le C h^d \sum_{\bi \in I} \| \bu^{h; 0} \|_2^p
		\end{equation}
		for all $n \in \{0, 1, \dots, N\}$.
		
		\item For all $r \leq h$, there exist $\alpha,\,C>0$ (independent of $h$) such that			
		\begin{equation}
		\label{hyp:unifbd}S^p_r (\mu^{h}, T) \leq C r^{\alpha},
		\end{equation}
		with $S^p_r(\mu, T)$ the time-averaged structure function
		\begin{equation*}
		S^p_r(\mu,T) := \left( \int_0^T \int_{L^p(D; \R^d)} \int_D \fint_{B_r(x)} \| \bu(x) - \bu(y)\|_2^p \,dy \,dx\,d\mu_{t}(\bu)\,dt \right)^\frac{1}{p}.
		\end{equation*}
		
		\item Scaling assumption: for all $\ell > 1$, there exist $\bar{C}>0$, $\beta \leq \alpha$ (independent of $h$) such that 			
		\begin{equation}
		\label{hyp:scaling} S^p_{\ell r} (\mu^h, T) \leq \bar{C} \ell^{\beta} S^p_{r} (\mu^ h, T).
		\end{equation} 		
	\end{itemize}
	Then there exists a subsequence $h' \to 0$ such that the approximate statistical solutions $\mu^{h'}$ converge strongly to some $\mu \in {L^1_t(\Prob)}$, in the sense of Theorem \ref{thm:compactness}.
\end{theorem}

\begin{proof}
	Let us verify that the conditions of Theorem \ref{thm:compactness} are met; hence a converging subsequence exists. Uniform time-regularity is immediate from Lemma \ref{lemma:uniformtimecont}. 
	
	Let $\{h_k\}_{k\in \mathbb{N}} \subset \mathbb{R}$, with $\{h_k\}\searrow 0$. Let us denote $\bnu_k = (\nu^1_k, \nu^2_k, \dots)$ the time-parameterized correlation measure associated with $\mu^{h_k} = \mathcal{S}^{h_k}_t\# \mu$ (in the sense of Theorem \ref{thm:duality}). $L^p$-boundedness is immediate: since $\supp(\bar{\mu}) \subset B_R(0)$, we have
	\begin{align}
	\sup\limits_{k \to \infty} &\sup\limits_{t \in [0, T)} \int_D \langle \nu^1_{k;t,x}, \|\xi\|_2^p \rangle ~dx \nonumber \\
	&= \sup\limits_{k \to \infty} \max\limits_{0 \le n \le N(k)}  \int_{L^p(D; \R^d)} \sum_{\bi \in I} h_k^d\|\bu(x_\bi) \|_2^p ~d\mu_{t^n}^{h_k}(\bu) \nonumber  \\
	&\stackrel{\eqref{hyp:Lp}}{\le} \sup\limits_{k \to \infty} C \int_{L^p(D; \R^d)} \sum_{\bi \in I} h_k^d \|\bu(x_\bi)\|_2^p ~d\mu_0^{h_k}(\bu) \nonumber  \\
	& \le \sup\limits_{k \to \infty} C' \int_{L^p(D; \R^d)} \|\bu\|_{L^p(D; \R^d)}^p ~d\barmu(\bu) \le C' R^p < \infty \label{eq:useinitconv}
	\end{align}
	where the inequality that begins line \eqref{eq:useinitconv} is justified by Lemma \ref{lemma:initconv}, taking $G$ to be the identity function.

	For the second property in Theorem \ref{thm:compactness}, we need to show that
	\begin{equation}
	\label{eq:bvclaim}
	\lim\limits_{r \to 0} \limsup\limits_{k \to \infty} \int_0^{T'} \int_D \fint_{B_r(x)} \langle \nu_{k; t, x,y}^2, \|\xi_1 - \xi_2\|^p\rangle \,dy\,dx\,dt = 0,
	\end{equation}
	for all $T' \in [0, T)$. Choose $k \in \N$. For $r \le h_k$, from hypothesis \eqref{hyp:unifbd}, this is immediate. For $r > h_k$,
	choose $\ell > 1$ and $h < h_k$ such that $r = \ell h$.
	\[  S^p_{r}(\mu^{h_k}, T) = S^p_{\ell h}(\mu^{h_k}, T)  \stackrel{\eqref{hyp:scaling}}{\le} C \ell^{\beta}S^p_{h} (\mu^{h_k}, T) \stackrel{\eqref{hyp:unifbd}}{\le} C'' \ell^{\beta} h^{\alpha} \le C''' r^\beta, \]
	where the bound is independent of $k$. Hence, through claim \eqref{eq:bvclaim}, Theorem \ref{thm:compactness} guarantees convergence up to a subsequence to some time-parameterized correlation measure $\mu$.
\end{proof}

\begin{remark}
	Comparing with Lemma \ref{lemma:schemeprops}.\ref{lemma:schemeprops:L2}, hypothesis \eqref{hyp:Lp} holds (with $p=2$) for scheme \eqref{eq:scheme1}-\eqref{eq:scheme2}. Furthermore, it is not difficult to see, cf. \cite{FLMWSystems}, Theorem 4.1 or \cite{CPPThesis}, Theorem 4.4, that the total variation-like property in Lemma \ref{lemma:schemeprops}.\ref{lemma:schemeprops:tv} implies hypothesis \eqref{hyp:unifbd}.
	
	The scaling assumption \eqref{hyp:scaling} is considerably more difficult, and there is little expectation of being able to derive it. We remark that this assumption is a discrete version of the scaling hypothesis in Kolmogorov's K41 theory (see e.g. Chapter 6 in \cite{Frisch}); which is a common assumption in turbulence models, widely considered to hold in practice. Thus, the hypothesis appears natural.
\end{remark}

\begin{remark}
	\label{rem:structsuffcond}
	If there exist $K>0$ and $\gamma \in (0,1]$ such that for all $h$, $$S^p_r (\mu^{h}, T) \leq K r^{\gamma},$$ then assumption \eqref{hyp:unifbd} holds with $\alpha \coloneqq \gamma$, and assumption \eqref{hyp:scaling} holds with $\beta \coloneqq \gamma$, immediately allowing one to conclude convergence in the sense of Theorem \ref{thm:compactness}. In Section \ref{sec:examples}, we will show how this (sufficient, not necessary) condition can be experimentally verified, and that it holds for all cases considered.
\end{remark}

\subsection{Convergence to a statistical solution}
In the previous section, we have shown that under external assumptions and up to a subsequence, the Monte Carlo scheme in Algorithm \ref{algo:fkmt} converges to a correlation measure $\mu$. In this section, we will prove a result akin to the Lax--Wendroff theorem, i.e. that said limit is a statistical solution to the incompressible Euler equations.

\begin{theorem}\label{thm:lxw}	
	 Let $\bar{\mu} \in \Prob(L^2(D; \mathbb{R}^d))$ with bounded support, $\supp(\bar{\mu})\subset B_R(0)\subset L^2(D,\mathbb{R}^d)$ for some $R>0$, with $\bu \in \Ldiv(D; \U)$ $\barmu$-almost surely. Let $\mu^{h}$ be given by \eqref{eq:discrevol} for $h>0$, and assume that for some sequence $h_l\to0$, the sequence $\{\mu^{h_l}\}_{l\in\N}$ converges strongly to $\mu \in L^1_t(\Prob)$, in the sense of Theorem \ref{thm:compactness}.\ref{prop:strongconv_time}. Assume furthermore that there exists $\lambda > 0$ such that for all $n$, it holds that $\frac{\Dt^n}{h} \le \lambda$.
	Then $\mu$ is a statistical solution of the incompressible Euler equations.
\end{theorem}

\begin{proof}
	Let $\mu$ be the limit measure of $\mu^\hl$ as $l \to \infty$. 
	$\mu$ is time-regular, as an immediate consequence of the uniform time-continuity of the approximations (Lemma \ref{lemma:uniformtimecont}), the convergence in Wasserstein norm (result \eqref{eq:convW1} in Theorem \ref{thm:compactness}), and Lemma \ref{lemma:limittimecontinuous}. Thus, we need to verify that in the limit, properties 1 and 2 from Definition \ref{def:statsol} are fulfilled. 
	
	Let us first verify the weak form of the incompressibility equation, \eqref{eq:statsol2}. This is relatively straightforward. 
	Fix $t \in [0, T)$. By Theorem \ref{thm:compactness}.\ref{prop:strongconv_time}, we have strong convergence for this observable. That is,
	\begin{align*}
	&\int_{L^2(D; \R^d)} \int_{D^2} \left( \bu(x_1) \otimes \bu(x_2)\right) : (\nabla \psi(x_1) \otimes \nabla \psi(x_2))\,dx_1 dx_2 d\mu_t(\bu) \\
	& = \lim_{l\to\infty} \int_{L^2(D; \R^d)} \int_{D^2} \left( \bu(x_1) \otimes \bu(x_2)\right) : (\nabla \psi(x_1) \otimes \nabla \psi(x_2))\,dx_1 dx_2 d\mu_t^\hl(\bu) \\
	& = \lim_{l\to\infty} \int_{L^2(D; \R^d)} \int_{D^2} \left( \mathcal{S}_t^\hl \barbu(x_1) \otimes \mathcal{S}_t^\hl\barbu(x_2)\right) : (\nabla \psi(x_1) \otimes \nabla \psi(x_2))\,dx_1 dx_2 d\widebar{\mu}^\hl(\barbu) \\	
	&= \lim_{l\to\infty} \int_{L^2(D; \R^d)} \left(  \int_{D} \mathcal{S}_t^\hl \barbu(x) \cdot \nabla \psi(x) \,dx\right)^2  d\widebar{\mu}^\hl(\barbu) = 0.
	\end{align*}
	To justify the last equality, let $\bar{\psi}$ be the cell averages of $\psi$. Recall that $\mathcal{S}_t^\hl \barbu$ is $\bar{\mu}$-a.s. discretely divergence-free, and through Lemma \ref{lemma:adj}, $ \int_{D} \mathcal{S}_t^\hl \barbu(x) \cdot \gradh \bar{\psi}(x) \,dx = 0$.
	As $\psi$ is smooth, some trivial Taylor analysis shows that $\|\gradh \bar{\psi} - \nabla \psi \|_{L^\infty} \to 0$ uniformly, and thus the limit is zero.

	Let us now show property 1. Let $\phi_1, \phi_2, \dots, \phi_k \in C^\infty([0,T)\times D, \mathbb{R}^d)$ with $\nabla \cdot \phi_i = 0$ for all $i \in \{1, \dots, k\}$; we must verify that
	\begin{gather*} 
	\begin{aligned}
	\int_0^T  &\int_{L^2(D; \R^d)}\int_{D^k} \bigg\{ \left[ \bu(t,x_1) \otimes \bu(t,x_2) \otimes \cdots \otimes  \bu(t,x_k) \right] : \partial_t {\Phi} (t, \bx)
	\\
	&\qquad 
	+ \sum_{\alpha=1}^k \left[  \bu(t, x_1) \otimes \cdots (\bu(t, x_\alpha) \otimes \bu(t, x_\alpha)) \otimes  \cdots \otimes \bu(t, x_k) \right] : \nabla_{{x}_i} {\Phi}(t, \bx) \bigg\}
	~ d\bx  \, d\mu_t(\bu) \, dt
	\\
	& + \int_{L^2(D; \R^d)} \int_{D^k} \left[ \barbu(x_1) \otimes \cdots \otimes \barbu(x_k)\right] : {\Phi}(0,\bx) \, d\bx \,d\bar{\mu}(\barbu)
	=
	0,
	\end{aligned}
	\end{gather*}
	where
	\[ \bx := (x_1, x_2, \dots, x_k) \in D^k, \qquad \Phi(t, \bx) := \phi_1(t, x_1) \otimes \phi_2(t, x_2) \otimes \cdots \otimes \phi_k(t, x_k).\]
	
	
	Fix $l\in \N$. For $h_l$, $\Phldiv$ induces in $\Ghldiv$ the pushforward measure \[\bar{\mu}^\hl := \Phldiv \# \bar{\mu}.\]

	Let $\barbuhl \in \supp(\bar{\mu}^\hl)$ be a discretely divergence-free initial condition. For all $t \in [0, T]$, let
	\begin{equation*}
	\bu^\hl(t) := \mathcal{S}^\hl_t(\barbuhl);
	\end{equation*}
	 and for all $i \in \{1, \dots, k\}$, let $\phi^{h_l}_i(t) := \Phdiv \circ \phi_i(t, \cdot) \in \Ghdiv$. Lemma \ref{lemma:Phdivprops}.\ref{lemma:Phdivprops:smalldiv} guarantees that as $\hl\to 0$, $\|\phihl(t) - \phi(t, \cdot)\|_{L^2(D; \U)}$ uniformly tends to zero. Let $\Phihl$ be the spatially piecewise constant tensor function given, for all $\bj = (j_1, \dots, j_k)$, by
	\[\Phihl_\bj(t, \bx) := (\phihl_1)_{j_1} (t, x_1) (\phihl_2)_{j_2} (t, x_2) \cdots (\phihl_k)_{j_k} (t, x_k).\]
	
	Denote for convenience $G(\bu, \bx, t) := \bu(t, x_1) \otimes \bu(t, x_2) \otimes \cdots \otimes  \bu(t, x_k)$; also $\bar{\mu}^\hl$-a.s. piecewise constant in space. Then
	
	\begin{align}
	&\int_0^T \int_{L^2(D; \R^d)} \int_{D^k} G(\bu, \bx, t) : \partial_t {\Phihl} ~ d\bx \, d\mu_t^\hl(\bu)\, dt \nonumber \\
	&= \int_{L^2(D; \R^d)} \int_{D^k} \sum_{n=0}^{N(\hl)-1}\int_{t^n}^{t^{n+1}}  G(\mathcal{S}^\hl_t (\bu), \bx, t) : \partial_t {\Phihl} \, dt\, d\bx \, d\barmu^\hl(\bu). \label{eq:toibptime}
	\end{align}
	
	Note that for all $x \in D$, $\bu(\cdot, x)$ is $\barmu^\hl$-a.s. linear in $[t^n, t^{n+1}]$, so integrate by parts eq. \eqref{eq:toibptime} in time,	
	\begin{gather}
	\begin{aligned}
	\label{eq:summary1}
	0 = &\int_0^T \int_{L^2(D; \R^d)} \int_{D^k} G(\bu, \bx, t) : \partial_t {\Phihl}(t,x) ~ d\bx \, d\mu_t^\hl(\bu)\, dt \\
	&\qquad + \int_{L^2(D; \R^d)} \int_{D^k} \left[ \barbu(x_1) \otimes \cdots \otimes \barbu(x_k)\right] : {\Phihl}(0,\bx)
	~ d\bx \, d\bar{\mu}^\hl(\barbu) \\
	&\qquad + \int_{L^2(D; \R^d)} \int_{D^k} \sum_{n=0}^{N(\hl)-1}\int_{t^n}^{t^{n+1}}  \left[ \partial_t G(\mathcal{S}^\hl_t (\bu), \bx, t) \right] : \Phihl(t,\bx) \, dt\, d\bx \, d\barmu^\hl(\bu).
	\end{aligned}
	\end{gather}
	
	Observe now that, as $\bu^\hl(t, x) = \mathcal{S}^\hl_t \left( \barbu^\hl(x) \right)$,
	\begin{equation}
	\label{eq:dtuphi}
	\partial_t G(\mathcal{S}^\hl_t \circ \bu, \bx, t) : \Phihl(t, \bx)
	= \sum_{\alpha=1}^k \left(\partial_t \bu^\hl(t, x_\alpha) \cdot \phihl_\alpha(t, x_\alpha)\right) \left(\prod_{\substack{\beta=1 \\ \beta\neq \alpha}}^{k} \bu^\hl(t, x_\beta) \cdot \phihl_\beta(t, x_\beta)\right).
	\end{equation}
	
	For $t \in (t^n, t^{n+1})$ for any $n$, it holds $\mu^{h_l}_t$-a.s. that:
	\begin{equation*}
	\partial_t \bu^\hl(t) +  \bm{C}^\hl(\bu^{\hl;n}, \bhlbu) - \bD^\hl(\bhlbu) - \gradhl \psi^{\hl;n} = 0,
	\end{equation*}
	with all terms and operators as defined in Section \ref{sec:scheme}. The integral of eq. \eqref{eq:dtuphi} (as in \eqref{eq:summary1}) can be rewritten with some reordering as
	\begin{align*}
	&\int_{L^2(D; \R^d)} \int_{D^k} \sum_{n=0}^{N(\hl)-1}\int_{t^n}^{t^{n+1}}  \left[ \partial_t G(\mathcal{S}^\hl_t(\bu), \bx, t) \right] : \Phihl(t,\bx) \, dt\, d\bx \, d\barmu^\hl(\bu) \\
	&= \sum_{\alpha=1}^k \int_{L^2(D; \R^d)}  \sum_{n=0}^{N(h)-1} \int_{t^n}^{t^{n+1}}   \left[\sumbi h_l^d  \left(\bD^\hl(\bhlbu)_\bi + \gradhl \psi^{\hl;n}_\bi -\bm{C}^\hl(\bu^{\hl;n}, \bhlbu)_\bi \right) \cdot (\phihl_\alpha)_\bi(t)\right] \\
	&\hspace{10em} \prod_{\substack{\beta=1 \\ \beta\neq \alpha}}^k \left[ \sumbi \hl^d \bu^\hl_\bi(t) \cdot (\phihl_\beta)_\bi(t) \right] ~dt \, d\mu^\hl(\bu). 
	\end{align*}

	Let us now simplify the expression above. Due to the $L^2$ bound of the scheme, and the convergence in $L^\infty$ of $\phihl$ to $\phi$ (Lemma \ref{lemma:Phdivprops}.\ref{lemma:Phdivprops:Linfbd}), it is immediate that for all $\beta \in \{1, \dots, d\}$, and for all $t \in [0, T)$, uniformly,
	\[ \sumbi \hl^d \bu^\hl_\bi(t) \cdot (\phihl_\beta)_\bi(t) - \int_D \bu^\hl(t, x) \cdot \phi_\beta(t,x) \,dx \stackrel{l\to\infty}{\longrightarrow} 0. \]
	
	Lemma \ref{lemma:schemeprops}.\ref{lemma:schemeprops:D} guarantees that $\bar{\mu}^\hl$-a.s., and for any $1 \le \alpha \le k$,
	\[ \lim_{l \to \infty} \sum_{n=0}^{N(h)-1} \int_{t^n}^{t^{n+1}}\sumbi \hl^d \bD^\hl(\bhlbu)_\bi  \cdot (\phihl_\alpha)_\bi(t) = 0, \]
	and this convergence is uniform. A direct application of Cauchy-Schwarz's inequality, together with the $\bar{\mu}$-a.s. $L^2$-boundedness of $\bu$ and $L^\infty$ boundedness of $\phi_i$ for all $i$ (cf. Lemma \ref{lemma:Phdivprops}.\ref{lemma:Phdivprops:Linfbd}), show that $\bar{\mu}$-a.s.,
	\[ \lim_{l \to \infty} \sum_{n=0}^{N(h)-1} \int_{t^n}^{t^{n+1}} \sum_{\alpha=1}^k  \left[ \sumbi h_l^d  \bD^\hl(\bhlbu)_\bi \cdot (\phihl_\alpha)_\bi(t)\right]\prod_{\substack{\beta=1 \\ \beta\neq \alpha}}^k \left[ \sumbi \hl^d \bu^\hl_\bi(t) \cdot (\phihl_\beta)_\bi(t) \right] ~dt \, d\mu^\hl(\bu) = 0.\]
	
	Furthermore, as $\phihl_i$ is $\bar{\mu}^\hl$-a.s. discretely divergence-free, $\bar{\mu}$-a.s., cf. Lemma \ref{lemma:adj}, clearly
	\[ \sumbi \gradhl \psi^{\hl;n}_\bi \cdot (\phihl_\alpha)_\bi(t) = 0.\]

	Recall that $\bC$ is defined as a difference of fluxes $\bFm$, which can be easily seen to be consistent, in the sense that $\bFm(\bu,\bu,\bu,\bu) = u_m \bu$, and Lipschitz-continuous (under an assumption of $L^2$-boundedness of $\bu$, cf. Lemma \ref{lemma:schemeprops}.\ref{lemma:schemeprops:L2}). A cumbersome but straightforward argument, cf. \cite{CPPThesis}, pp. 74-75, leveraging the properties listed in Lemmas \ref{lemma:schemeprops} and \ref{lemma:Phdivprops}, shows the following: for $t \in [t^n, t^{n+1})$,
	\begin{align*}&\lim_{l\to\infty} h_l^d \sumbi -\bm{C}^\hl(\bu^{\hl;n}, \bhlbu)_\bi  \cdot (\phihl_\alpha)_\bi(t) \\
	& \qquad = \lim_{l\to\infty}\int_D \sum_{m=1}^d (u_m)^{\hl;n}(x) \bu^{\hl;n}(x) \cdot \partial_{x_m} (\phi_\alpha)(t,x) \,dx \\
	& \qquad = \lim_{l\to\infty}\int_D \left[ \bu^{\hl;n} \otimes \bu^{\hl;n} (x) \right] : \nabla (\phi_\alpha)(t,x) \,dx \end{align*}

	Thus, substituting into eq. \eqref{eq:summary1} and taking the limit as $l \to \infty$, we find
	
	\begin{align*}
	0 = 
	\lim_{l \to \infty} \Bigg\{ &\int_{L^2(D; \R^d)}\int_0^T  \int_{D^k} \left[ \bu(t, x_1) \otimes \bu(t, x_2) \otimes \cdots \otimes  \bu(t, x_k) \right] : \partial_t {\Phihl}
	~ d\bx \, dt \, d\mu^\hl(\bu)\\
	&\qquad+ \int_{L^2(D; \R^d)}  \int_{D^k} \left[ \barbu(x_1) \otimes \cdots \otimes \barbu(x_k)\right] : {\Phihl}(0,\bx)
	~ d\bx \, d\bar{\mu}^\hl(\barbu)\\
	&\qquad +  \int_{L^2(D; \R^d)}  \int_0^T \int_{D^k} \sum_{\alpha=1}^{k} \left[ \bu(t, x_1) \otimes \cdots \otimes (\bu(t^n, x_\alpha) \otimes \bu(t^n, x_\alpha)) \otimes \cdots \otimes \bu(t, x_k) \right] \\
	&\hspace{12em}: \nabla_{x_\alpha} \Phi(t,\bx) \,d\bx \,dt \, d\mu^\hl(\bu) \Bigg\}
	\end{align*}
	
	Note that the term $(\bu(t^n, x_\alpha) \otimes \bu(t^n, x_\alpha))$ remains not parameterized with respect to $t$. However, it is not difficult to show, \cite{CPPThesis}, Corollary A.19, that the integral of $(\bu(t^n, x_\alpha) \otimes \bu(t^n, x_\alpha)) - (\bu(t, x_\alpha) \otimes \bu(t, x_\alpha))$ vanishes in the limit. Hence, additionally replacing $\phihl$ and $\Phihl$ by their continuous counterparts, again through the $L^2$-boundedness of $\bu$ and the $L^\infty$-convergence of $\phihl$ to $\phi$ (Lemma \ref{lemma:Phdivprops}.\ref{lemma:Phdivprops:Linfbd}) we conclude that
	
	\begin{align*}
	0 
	&= \lim_{l \to \infty} \Bigg\{ \int_{L^2(D; \R^d)}\int_0^T  \int_{D^k} \left[ \bu(t, x_1) \otimes \bu(t, x_2) \otimes \cdots \otimes  \bu(t, x_k) \right] : \partial_t {\Phi}
	~ d\bx \, dt \, d\mu^\hl(\bu)\nonumber\\
	&\qquad+ \int_{L^2(D; \R^d)}  \int_{D^k} \left[ \barbu(x_1) \otimes \cdots \otimes \barbu(x_k)\right] : {\Phi}(0,\bx)
	~ d\bx \, d\bar{\mu}^\hl(\barbu)\nonumber\\
	&\qquad +  \int_{L^2(D; \R^d)}  \int_0^T \int_{D^k} \sum_{\alpha=1}^{k} \left[ \bu(t, x_1) \otimes \cdots \otimes (\bu(t, x_\alpha) \otimes \bu(t, x_\alpha)) \otimes \cdots \otimes \bu(t, x_k) \right] \nonumber\\
	&\hspace{12em}: \nabla_{x_\alpha} \Phi(t,\bx) \,d\bx \,dt \, d\mu^\hl(\bu) \Bigg\} \nonumber\\
	&= \int_{L^2(D; \R^d)}\int_0^T  \int_{D^k} \left[ \bu(t, x_1) \otimes \bu(t, x_2) \otimes \cdots \otimes  \bu(t, x_k) \right] : \partial_t {\Phi}
	~ d\bx \, dt \, d\mu(\bu)\nonumber\\
	&\qquad+ \int_{L^2(D; \R^d)}  \int_{D^k} \left[ \barbu(x_1) \otimes \cdots \otimes \barbu(x_k)\right] : {\Phi}(0,\bx)
	~ d\bx \, d\bar{\mu}(\barbu)\nonumber\\
	&\qquad +  \int_{L^2(D; \R^d)}  \int_0^T \int_{D^k} \sum_{\alpha=1}^{k} \left[ \bu(t, x_1) \otimes \cdots \otimes (\bu(t, x_\alpha) \otimes \bu(t, x_\alpha)) \otimes \cdots \otimes \bu(t, x_k) \right]\nonumber \\
	&\hspace{14em}: \nabla_{x_\alpha} \Phi(t,\bx) \,d\bx \,dt \, d\mu(\bu). \nonumber
	\end{align*}
	
	In the last step, we have used the assumption of convergence in the sense of Theorem \ref{thm:compactness}.\ref{prop:strongconv_time}. That is: the limit measure verifies the momentum equation \eqref{eq:statsol1}, and is indeed, a statistical solution of the incompressible Euler equations.
	
\end{proof}

\begin{remark}
	In fact, under the conditions of Theorem \ref{thm:lxw}, the limit $\mu_t$ can be seen to be a \emph{dissipative statistical solution} in the sense of Def. \ref{def:dissipativestatsol}. The fundamental observation is that energy is an admissible observable in the sense of Theorem \ref{thm:compactness}.\ref{prop:strongconv_time}; the proof for the convex splitting is tedious but straightforward. Cf. Remark 4.6 in \cite{LMP1}. In particular, then Theorems \ref{thm:localeu} and \ref{thm:wsu} for well-posedness hold.
\end{remark}

\begin{remark}
	The results proved here extend in a sense those of \cite{Chen2012}. Under a similar assumption to \eqref{hyp:scaling} they prove strong convergence of the numerical approximations, understood as a vanishing (numerical) viscosity sequence, to a weak solution of the incompressible Euler equations. If one assumes that this solution is unique, then combining the results in this section with those of Chen and Glimm, one could conclude that Algorithm \ref{algo:fkmt} approximates a random family of weak solutions to the incompressible Euler equations. However, crucially, the approach described here does not assume a unique weak solution exists, and produces meaningful results regardless. In particular in section \ref{ss:sls} we will display an example that suggests convergence of the algorithm to a statistical solution for the discontinuous shear layer, where uniqueness of weak solutions, at least in an unconstrained setting, is known to fail.
\end{remark}

\section{Numerical examples}
\label{sec:examples}
The Monte Carlo algorithm \ref{algo:fkmt} is eminently practical; we present here some novel numerical examples, for well-known test cases, computed with the implementation of the scheme in \cite{luqness}.

\subsection{Setting and notation}
Regarding the numerical parameters for scheme \eqref{eq:scheme1}-\eqref{eq:scheme2}, we consider $\theta = 1$, a coefficient for numerical viscosity $\epsilon= 0.1$, and a time-step that satisfies a CFL-type condition for the underlying advection problem with CFL number $\lambda = 0.5$.


Let $N, M \in \mathbb{N}$, $T \in \mathbb{R}$. We will denote an \textbf{ensemble} of $M$ simulations at resolution $N \times N$ up to time $T$ by
\[U^{N,M,T} := \{U^{N,m,T}\}_{m=1}^M, \]
where each $U^{N,m,T}$
, dubbed \textbf{sample}, is a separate realization of a Cauchy problem on an $N \times N$ grid, up to time $T$, as described in Algorithm \ref{algo:fkmt}. $U^{N,m,T}$ can be considered to be an element $\bu$ in the support of a discrete probability measure $\mu^h$ constructed by Algorithm \ref{algo:fkmt}.

\begin{remark}
	\label{rem:ensembleaspdf}
	Fixing $t$, an ensemble $U^{N,M,t}$ can be trivially identified with a probability distribution $\mu_t \in \Prob(L^p(D; \R^d))$, which we term its \emph{empirical measure}:
	\begin{equation}\label{eq:discrempiricalmeasure} \mu_t^{N,M} = \frac{1}{M}\sum_{m=1}^M \delta_{U^{N,m,t}}. \end{equation}
\end{remark}

A priori, the number $M$ of samples in Algorithm \ref{algo:fkmt} and the spatial resolution $N$ of the finite volume scheme are unrelated. However, to ensure convergence as the mesh is refined, we make the choice of equating $M = N$. Even in two dimensions, the cost of higher-resolution simulations makes it prohibitive to scale $M$ like $N^2$ (as a naive Monte Carlo approximation would have); we remark however that the convergence proven in Theorems \ref{thm:relativeconv} and \ref{thm:lxw} is without a rate. This choice is frequent in the literature, see e.g. \cite{FLM18}, \cite{FLMWSystems}, \cite{LMP1}... We provide some numerical evidence for this choice in Section \ref{sss:MCerror}.

In this section we will often refer, for a time $t$, to the \emph{sample mean} and \emph{variance operators}, defined as those of the empirical measure.

For the discretization of the structure function,
we will follow the choices of \cite{LyePhd} and take 
\begin{align*}
\int_\domain\fint_{B_r(x)} \|\bu(y) - &\bu(x)\|^p \,dy\, dx \\
\approx \frac{h^2}{l^2} \sum_{i=1}^N \sum_{j=1}^N \bigg[ &\sum_{k=-l+1}^l \sum_{n=-l+1}^l \left( \|U_{i+k,j+n} - U_{i,j}\|^p \right) \\
&+ \frac{1}{2} \sum_{k=-l+1}^l \left(\| U_{i+k,j-l} - U_{i,j}\|^p + \| U_{i+k,j+l} - U_{i,j}\|^p \right) \\
&+ \frac{1}{2} \sum_{n=-l+1}^l \left(\| U_{i-l,j+n} - U_{i,j}\|^p + \| U_{i+l,j+n} - U_{i,j}\|^p \right) \\
&+ \frac{1}{4} \sum_{\alpha, \beta= 0}^1 \|U_{i + (-1)^\alpha l, j + (-1)^\beta l} - U_{i,j}\|^p \bigg].
\end{align*}

For reasons of computational efficiency, we will present structure functions at fixed times, i.e. for a given $t \in [0, T)$ we compute
\begin{equation*}
S^p_{r,t}(\bnu) := \left(\int_{L^p_x} \int_\domain\fint_{B_r(x)} \|\bu(y) - \bu(x)\|^p \,dy\, dx\, d\mu_t(u)\right)^\frac{1}{p}.
\end{equation*}

As will be shown in the sequel, we observe uniform bounds for all time points considered for all numerical cases. This is strong evidence for uniform boundedness for the time-integrated version, and thus for convergence (up to a subsequence) of the scheme; compare to Remark \ref{rem:structsuffcond}.

\subsection{Double shear layers}
\label{ss:sls}
Here we present two tests in a well-known family, the two-dimensional double \emph{shear layer}.

We dub \emph{smooth shear layer} the following initial datum: fix $\rho > 0$, and consider the initial condition for $(x,y) \in \T^2$:
\[ \tilde{u}_0(x,y) = \begin{cases}
\tanh\left( (y - 0.25)/\rho \right) &\mbox{ if } y \le 0.5, \\
\tanh\left( (0.75 - y)/\rho \right) &\mbox{ otherwise; }
\end{cases}\qquad \qquad \tilde{v}_0(x,y) = 0. \]

The \emph{discontinuous shear layer} is the pointwise a.e. limit of the above when $\rho \to 0$; that is, 
\[ \tilde{u}_0(x,y) = \begin{cases}
1 &\mbox{ if } y \in (0.25, 0.75), \\
-1 &\mbox{ otherwise};
\end{cases}\qquad  \qquad \tilde{v}_0(x,y) = 0. \]

The \emph{smooth} initial datum satisfies the conditions for local existence of classical solutions of the incompressible Euler equations. Therefore, at least for a short time horizon, one can expect the existence of a unique strong solution; and thus, due to weak-strong uniqueness, see \cite{Brenier11}; also cf. Theorem \ref{thm:wsu}, as long as the classical solution is defined, there exists a unique measure-valued solution, and that is the atomic measure corresponding to the classical solution.

Conversely, the \emph{discontinuous} initial datum has vorticity which is a measure without distinguished sign, so it falls outside of Delort's class; in fact it is explicitly known, \cite{Szekelyhidi11}, to belong to the class of ``wild initial data'', i.e. infinitely many (admissible) weak solutions to this problem exist.

Fix now $0 < \gamma \in \R$, $K \in \N$ even, and consider a random perturbation function 
\begin{align}
\begin{split}
f_\gamma &: \Omega \times \T^2 \to \T^2\\
f_\gamma(\omega; x,y) &= \left(x, y + \gamma \sum_{k=0}^{K/2} Y_{2k}(\omega) \sin(2 \pi (k+1) (x + Y_{2k+1}(\omega))) \right), \label{eq:pertfct}
\end{split}
\end{align}
with $Y_j \sim \mathcal{U} [-1,1]$ i.i.d. $\forall j$. Finally, set initial velocities
\[ (u_0, v_0)(\omega, \cdot) = (\tilde{u}_0, \tilde{v}_0) \circ f_\gamma(\omega; \cdot);\]
we recall that discrete divergence is projected out (with operator $\Phdiv$) before simulation starts. See Fig. \ref{fig:slpert} for an example of the effect of the perturbation on the initial condition.

Let us first remark that, for individual realizations of the Cauchy problem, the two variants of the initial datum exhibit drastically different evolutions. In Fig. \ref{fig:slCauchy} we present, at different points in time $t$, the Cauchy rates $\|\bu^{h}(t) - \bu^{\frac{h}{2}}(t) \|_{L^2}$ for discretizations of the same realization of the initial datum. Clearly as the mesh is refined, the smooth shear layer, at all times, appears to form a Cauchy sequence in $L^2$; the same is only true for the discontinuous version at $t=0$; for all later times, the Cauchy rate appears to flatten rapidly.

\begin{figure}
	\centering
	\begin{subfigure}{.3\textwidth}
		\centering
		\includegraphics[width=\linewidth]{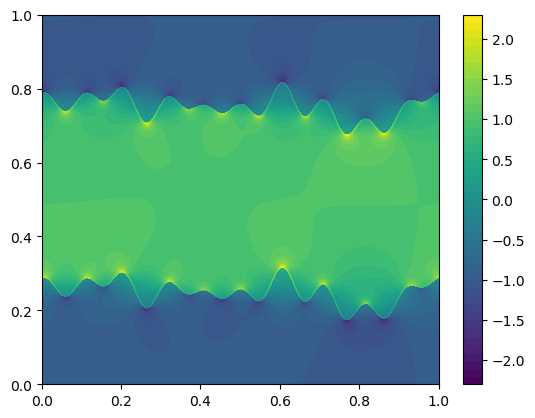}
		\subcaption{$\tilde{u}_0$}
	\end{subfigure}%
	\begin{subfigure}{.3\textwidth}
		\centering
		\includegraphics[width=\linewidth]{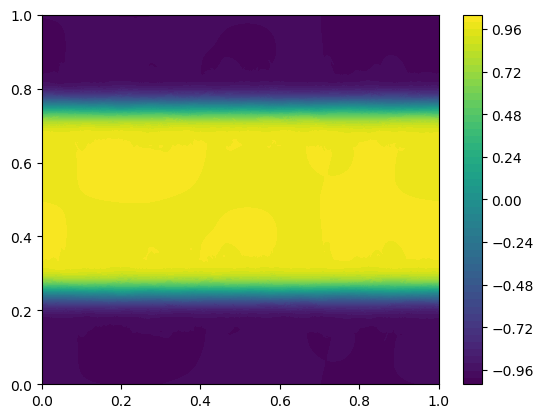}
		\subcaption{Mean (discont.)}
	\end{subfigure}%
	\begin{subfigure}{.3\textwidth}
		\centering
		\includegraphics[width=\linewidth]{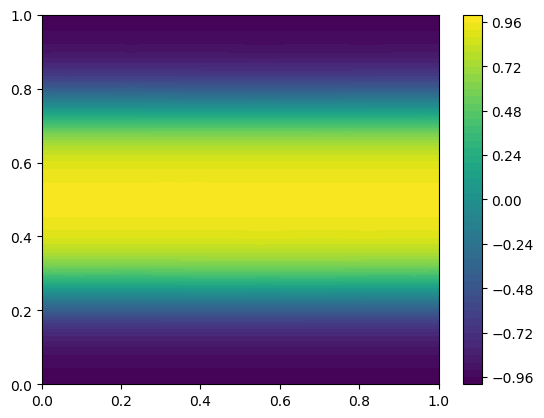}
		\subcaption{Mean (smooth)}
	\end{subfigure}

	\begin{subfigure}{.3\textwidth}
		\centering
		\includegraphics[width=\linewidth]{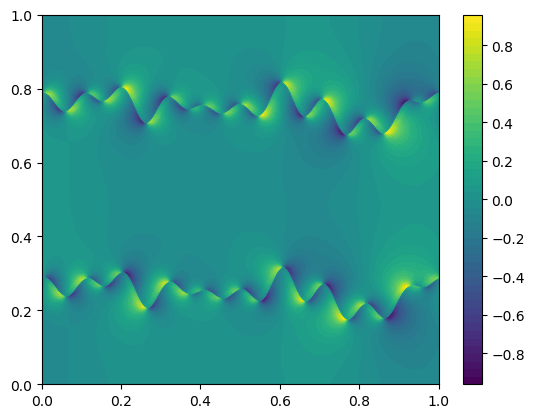}
		\subcaption{$\tilde{v}_0$}
	\end{subfigure}%
	\begin{subfigure}{.3\textwidth}
		\centering
		\includegraphics[width=\linewidth]{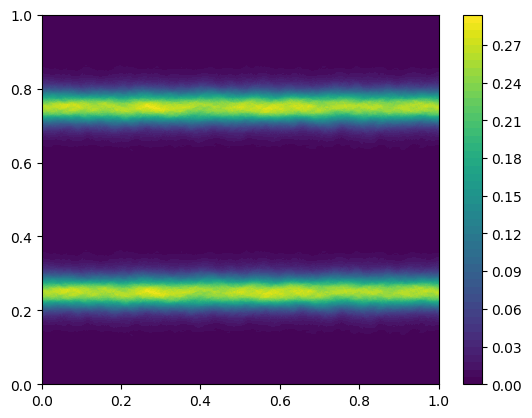}
		\subcaption{Variance (discont.)}
	\end{subfigure}
	\begin{subfigure}{.3\textwidth}
		\centering
		\includegraphics[width=\linewidth]{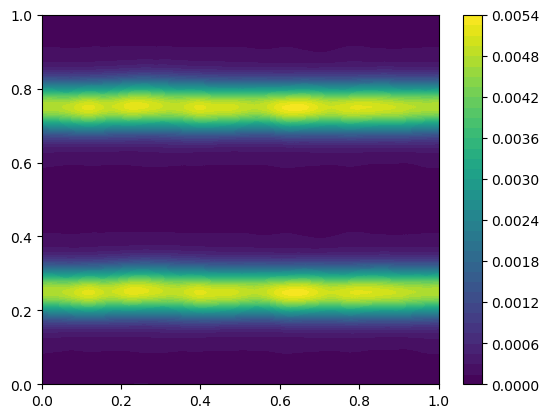}
		\subcaption{Variance (smooth)}
	\end{subfigure}
	
	\caption{Left: one realization of the random initial condition for discontinuous shear layer. Center: mean (top) and variance (bottom) for $\tilde{u}_0$, discontinuous shear layer. Right: mean and variance for $\tilde{u}_0$, smooth shear layer.}
	\label{fig:slpert}
\end{figure}

\begin{figure}
	\centering
	\begin{subfigure}{.35\textwidth}
		\centering
		\includegraphics[width=\linewidth]{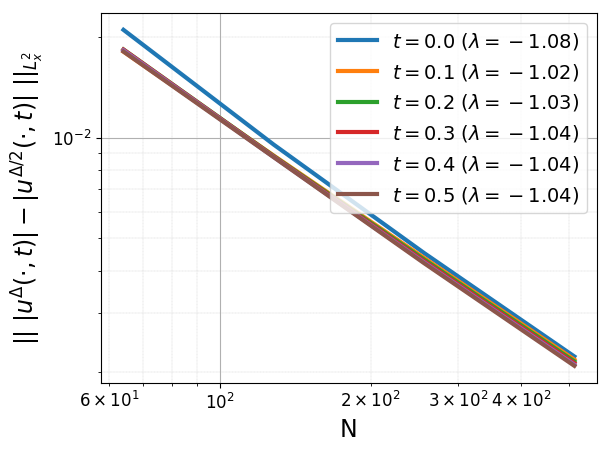}
		\subcaption{Smooth shear layer}
	\end{subfigure}%
	\begin{subfigure}{.35\textwidth}
		\centering
		\includegraphics[width=\linewidth]{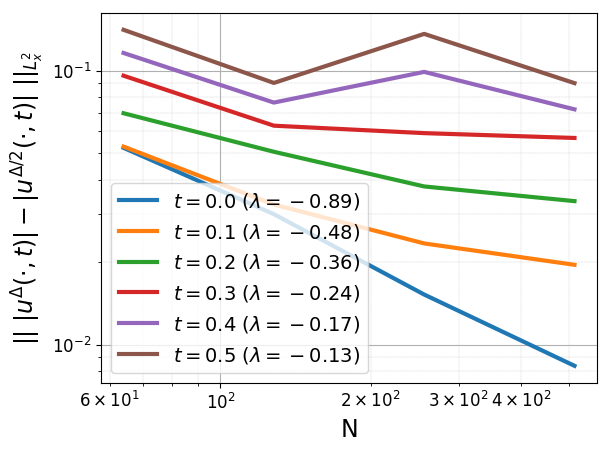}
		\subcaption{Discontinuous shear layer}
	\end{subfigure}	
	\caption{Cauchy rates for individual samples as time evolves, for smooth (left) and discontinuous (right) shear layer. $\lambda$ is the least square fit for the slope.}
	\label{fig:slCauchy}
\end{figure}

In Fig. \ref{fig:slstructnorm} we display the structure functions obtained for both cases at times $t \in \{0, 0.4\}$. The uniform boundedness appears obvious; in fact the initial data for the discontinuous shear layer gains regularity as the simulation evolves, due to the action of the numerical diffusion operator $\bD$ -- which, recall Lemma \ref{lemma:schemeprops}.\ref{lemma:schemeprops:D}, vanishes in the limit. Hence, through Theorems \ref{thm:relativeconv}-\ref{thm:lxw}, one can expect convergence (up to a subsequence) of Algorithm \ref{algo:fkmt}.

In fact, we find direct numerical evidence of this convergence for the sequence of resolutions considered. Let $\bnu_t^{N,N}$ the correlation measure associated, in the sense of Theorem \ref{thm:duality}, with the empirical measure $\mu_t^{N,N}$ in eq. \eqref{eq:discrempiricalmeasure}. In Fig. \ref{fig:slwass} we display the Cauchy rates for an approximation to the Wasserstein distances for the $k$-point marginals $\bnu_t^{k; N,N}$, i.e., we plot
\[ W_1\left(\bnu^{k; N,N}, \bnu^{k; 2N,2N}\right),\]
as a function of $N$, for $k \in \{1,2,3\}$. The results strongly suggest convergence (albeit noisy) of the marginal probability measures. In fact, as Theorem \ref{thm:compactness} indicates, observables such as mean and variance of the empirical measure, considered as functions in $L^2$, converge strongly; cf. Fig. \ref{fig:slmeanvar}.

\begin{figure}	
	\centering
	\begin{subfigure}{0.3\textwidth}
		\centering
		\includegraphics[width=\linewidth]{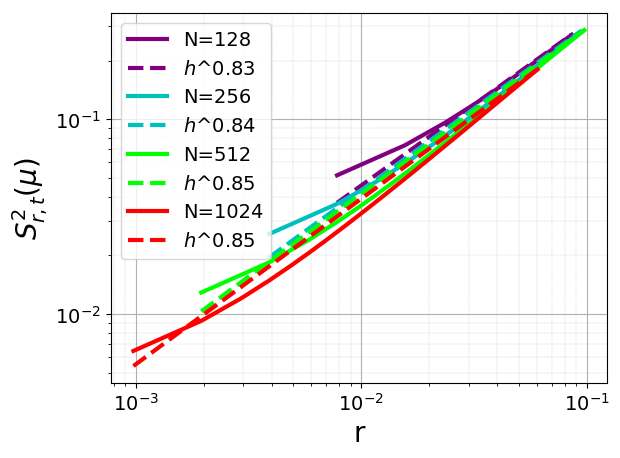}
	\end{subfigure} %
	\begin{subfigure}{0.3\textwidth}
		\centering
		\includegraphics[width=\linewidth]{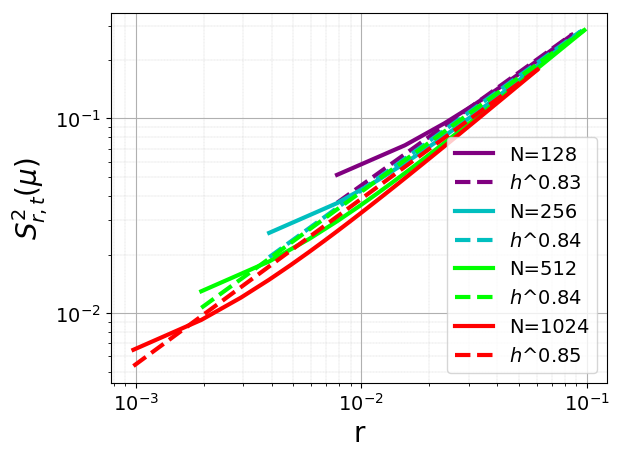}
	\end{subfigure}

	\begin{subfigure}{0.3\textwidth}
		\centering
		\includegraphics[width=\linewidth]{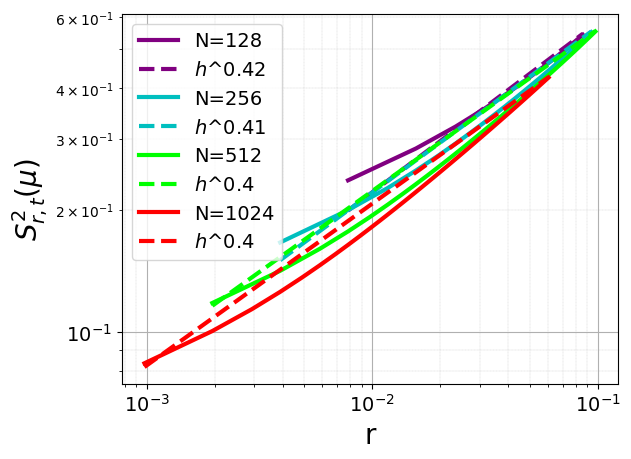}
	\end{subfigure} %
	\begin{subfigure}{0.3\textwidth}
		\centering
		\includegraphics[width=\linewidth]{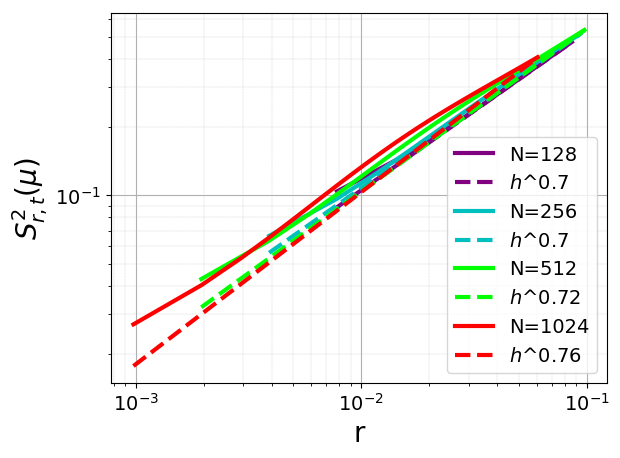}
	\end{subfigure}
	
	\caption{Structure functions for velocity vector, $p=2$, for smooth (top) and discontinuous shear layer (bottom). Left, $t=0$; right, $t=0.4$.}
	\label{fig:slstructnorm}
\end{figure}

\begin{figure}
	\centering
	\begin{subfigure}{.3\textwidth}
		\centering
		\includegraphics[width=\linewidth]{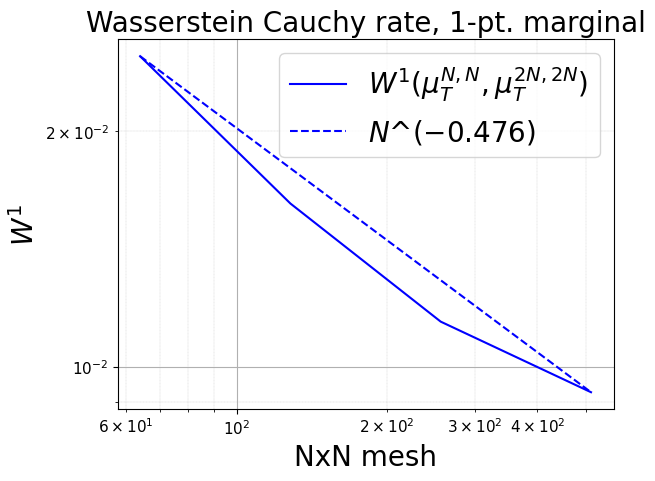}
		\subcaption{$k=1$}
	\end{subfigure}%
	\begin{subfigure}{.3\textwidth}
		\centering
		\includegraphics[width=\linewidth]{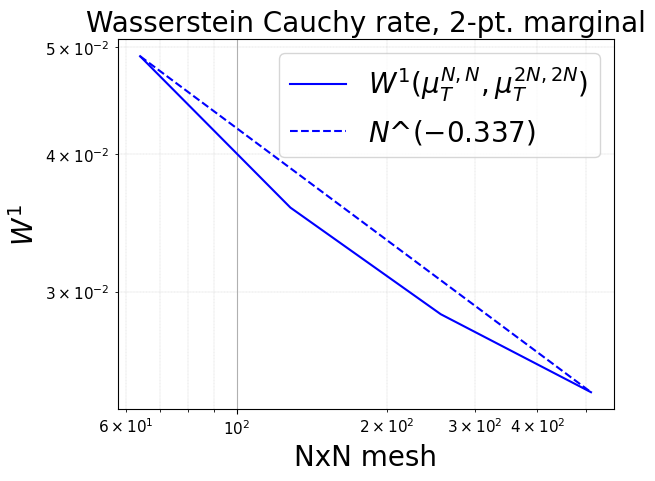}
		\subcaption{$k=2$}
	\end{subfigure}%
	\begin{subfigure}{.3\textwidth}
		\centering
		\includegraphics[width=\linewidth]{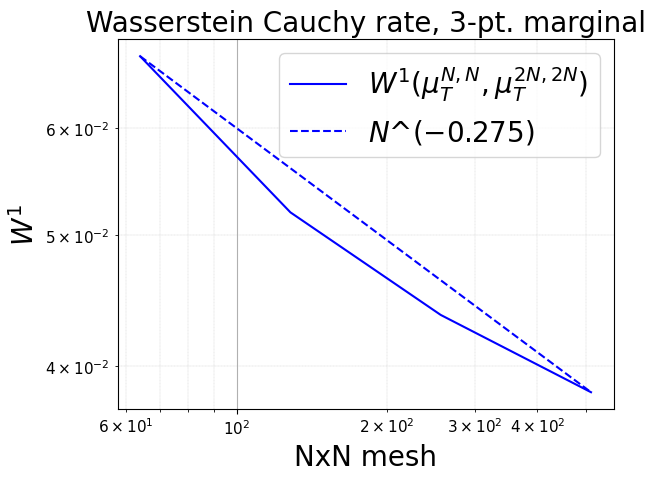}
		\subcaption{$k=3$}
	\end{subfigure}
	
	\begin{subfigure}{.3\textwidth}
		\centering
		\includegraphics[width=\linewidth]{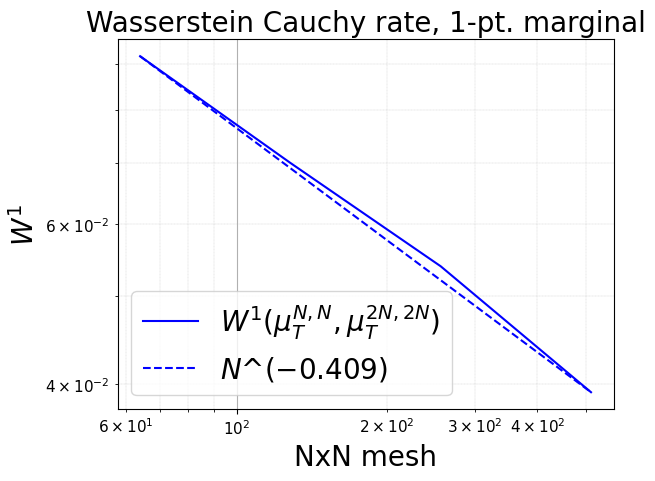}
		\subcaption{$k=1$}
	\end{subfigure}%
	\begin{subfigure}{.3\textwidth}
		\centering
		\includegraphics[width=\linewidth]{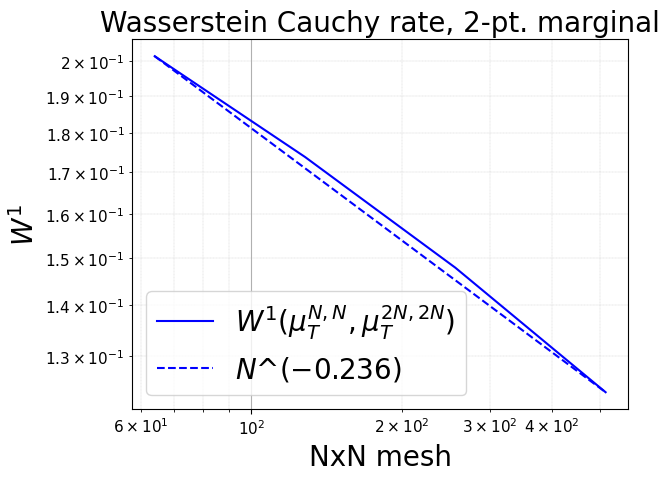}
		\subcaption{$k=2$}
	\end{subfigure}
	\begin{subfigure}{.3\textwidth}
		\centering
		\includegraphics[width=\linewidth]{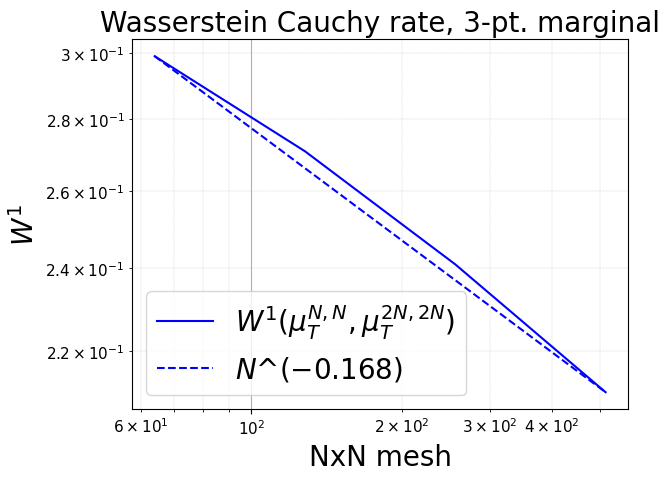}
		\subcaption{$k=3$}
	\end{subfigure}
	
	\caption{Cauchy rates for Wasserstein distances for the velocity vector field, (left to right) 1-, 2- and 3-point marginals. Smooth (top) and discontinuous (bottom) shear layer; $t=0.4$.}
	\label{fig:slwass}
\end{figure}

\begin{figure}
	\centering
	\begin{subfigure}{.35\textwidth}
		\centering
		\includegraphics[width=\linewidth]{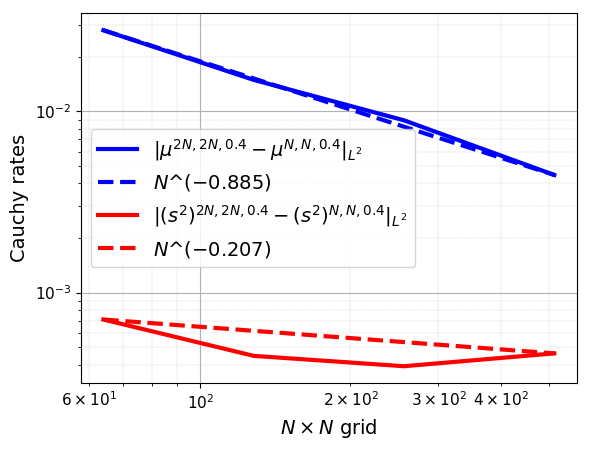}
		\subcaption{Smooth shear layer}
	\end{subfigure}%
	\begin{subfigure}{.35\textwidth}
		\centering
		\includegraphics[width=\linewidth]{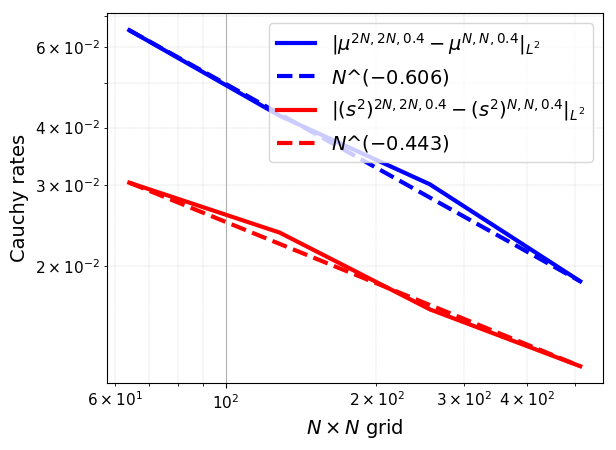}
		\subcaption{Discontinuous shear layer}
	\end{subfigure}	
	\caption{Cauchy rates for mean and variance for the smooth (left) and discontinuous (right) shear layer.}
	\label{fig:slmeanvar}
\end{figure}

We believe this test case paints a convincing picture of the advantage of the framework of statistical solutions: weak, deterministic solutions appear an unsuitable framework for numerics (Fig. \ref{fig:slCauchy}), at least for cases where no unique solution exists. Conversely, we obtain solid experimental evidence that assumption \eqref{hyp:scaling}, necessary for convergence of statistical solutions, holds for both cases (Fig. \ref{fig:slstructnorm}). Indeed, as predicted by Theorem \ref{thm:compactness} we observe strong convergence for observables such as mean and variance. Furthermore, we find evidence that the marginals converge, even for $k>1$; i.e. we do not only obtain convergence in the sense of measure-valued solutions, but properly for statistical solutions.

\subsubsection{Convergence to a non-trivial statistical solution}
In this section, we briefly present a result we find enlightening about the shear layer numerical experiment. As mentioned before, there is a fundamental difference in well-posedness for the smooth and discontinuous shear layer. However, in the numerical experiments above, we have found fairly analogous results for both, up to slightly faster rates of convergence for the smooth version. With the following example, we aim to show a key difference between the empirical measure obtained for both cases.

For this, recall \eqref{eq:pertfct}: we introduced a perturbation $f_\gamma$, with a ``small'' parameter $\gamma$. An interesting question is: what changes if one lets $\gamma \to 0$? Does one recover the unperturbed initial datum $(\tilde{u}_0, \tilde{v}_0)$, itself a steady state, as a solution to the problem? In other words: fix a high enough resolution $N$ (we take $N=1024$ here), and for $\gamma > 0$, denote $\mu_\gamma$ the empirical measure $\mu^{N,N}_t$ obtained with initial data distributed as $(\tilde{u}_0, \tilde{v}_0) \circ f_\gamma$. Consider the stationary solution $\bar{\bu}(t) \coloneqq (\tilde{u}_0, \tilde{v}_0)$ for all $t \ge 0$. Is it true that
\[ \lim_{\gamma \to 0} W_1(\mu_{\gamma,t}, \delta_{\bar{\bu}}) = 0 ? \]

The answer can be found in Fig. \ref{fig:slwassconv}. If one takes a very small value of $\gamma$ as a ``reference solution'', one can see that both examples appear to converge to the reference solution. However, the distance to $\delta_{\bar{\bu}}$ only appears to tend to zero for the smooth shear layer. This corresponds to the theoretical intuition: in that case, there exists a unique classical solution, and the only statistical solution is its atomic measure, due to weak-strong uniqueness. Conversely, infinitely many weak solutions exist for the discontinuous shear layer, and thus even as $\gamma \to 0$, we do not recover an atomic measure, but a non-trivial statistical solution.

\begin{figure}
	\centering
	
	\begin{subfigure}{.38\textwidth}
		\centering
		\includegraphics[width=\linewidth]{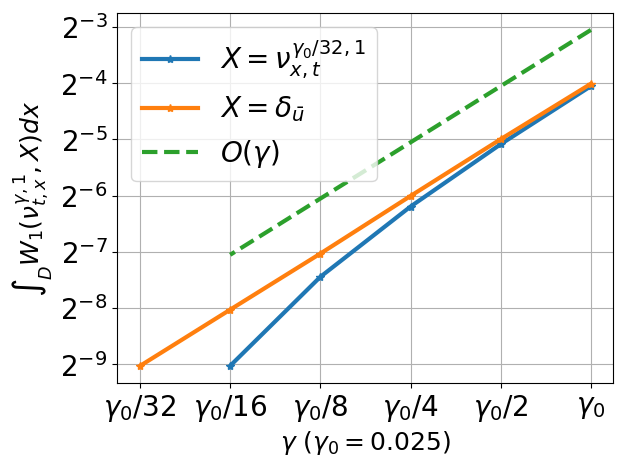}
	\end{subfigure}%
	\begin{subfigure}{.38\textwidth}
		\centering
		\includegraphics[width=\linewidth]{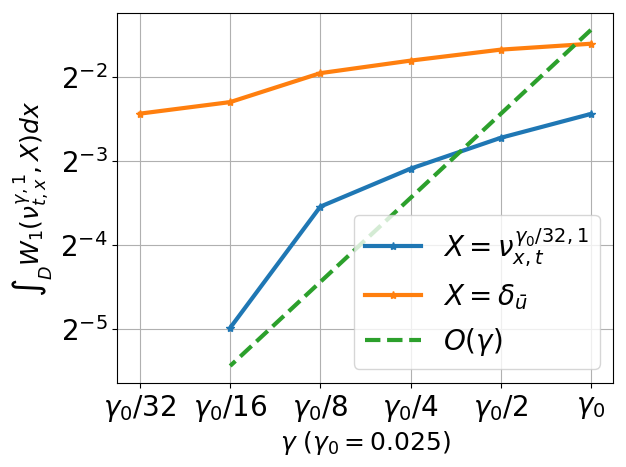}
	\end{subfigure}
	
	\caption{Wasserstein distances with respect to a very small perturbation (blue) and the unperturbed initial datum (orange) as the perturbation magnitude reduces. Smooth (left) and discontinuous (right) shear layer; 1-point marginals at time $t = 0.4$.}
	\label{fig:slwassconv}
\end{figure}

\subsubsection{Standard Monte-Carlo convergence}
\label{sss:MCerror}
A natural question is: how does Algorithm \ref{algo:fkmt}, and in particular the choice of increasing number of samples and spatial resolution in parallel, compare to ``standard'' Monte-Carlo, with a fixed spatial resolution and increasing the number of samples?

A priori, as the spatial error scales with $O(N^{-1})$ and the Monte-Carlo error wiuth $O(N^{-1/2})$ (as we choose $N$ samples), one can expect that the error will be dominated by the stochastic term, and the increase in resolution is moot.

In fact, this behavior is problem- and even observable-dependent. In Fig. \ref{fig:sl3d}, we can see that for the smooth shear layer at $T=0.4$, the mean is nearly homogeneous in the $x$ axis, as the limit must be. This suggests that Monte Carlo error is small in this scenario, and accuracy may only be improved by increasing the spatial resolution. Conversely, the variance for the same example is spatially noisy -- i.e. stochastic error still dominates, and the convergence can be expected to accelerate only with an increased number of samples, and spatial resolution will not have much of an effect.

\begin{figure}
	\centering
	\begin{subfigure}{.3\textwidth}
		\centering
		\includegraphics[width=\linewidth]{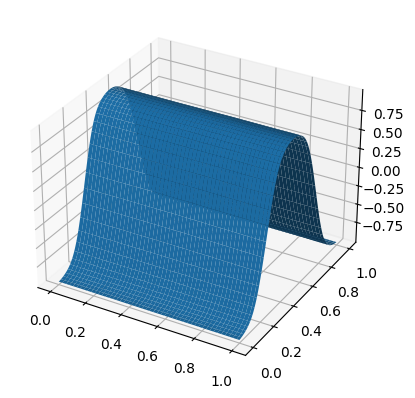}
		\subcaption{Mean}
	\end{subfigure}%
	\begin{subfigure}{.3\textwidth}
		\centering
		\includegraphics[width=\linewidth]{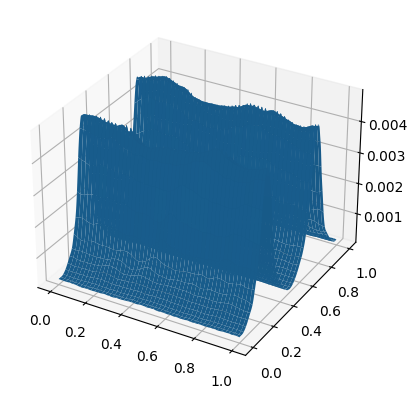}
		\subcaption{Variance}
	\end{subfigure}
	\caption{Observables for horizontal velocity, smooth shear layer, $T=0.4$, for $N=M=1024$.}
	\label{fig:sl3d}
\end{figure}

To experimentally verify this, we consider the following test: for the smooth and discontinuous vortex sheet, we fix a spatial resolution of $N=256$, and vary the number of samples $M=\{64, 128, \dots, 1024\}$. We analyze the Cauchy rates for mean and variance, keeping in mind that where Monte Carlo error dominates, one expects a convergence rate of $1/2$; and slower (to non-existent) where spatial error dominates, such as average of the smooth shear layer. This is precisely what we observe in Fig. \ref{fig:mcerror}.

This paints a clear picture that the heuristics of when refining spatially is optimal, and when to only increase the number of samples, are not trivial. Comparing to the Cauchy rates for the mean of the smooth shear layer in Fig. \ref{fig:slCauchy}, a convincing case appears for Algorithm \ref{algo:fkmt} outperforming fixed-resolution Monte Carlo.

A similar situation appears when studying the Cauchy rates for Wasserstein distances, see Fig. \ref{fig:mcwass}. Comparing to Fig. \ref{fig:slwass}, we observe a slightly decreased Cauchy rate for the smooth shear layer when maintaining spatial resolution fixed. In this case spatial error appeared to play a comparatively larger role than stochastic error, so fixing resolution is expected to hinder convergence. Conversely, for the discontinuous shear layer, little difference appears, as stochastic error appeared to dominate. We present only results for 1-point marginals for brevity, but the behavior is analogous for 2- and 3-point correlations as well.

\begin{figure}
	\centering
	\begin{subfigure}{.4\textwidth}
		\centering
		\includegraphics[width=\linewidth]{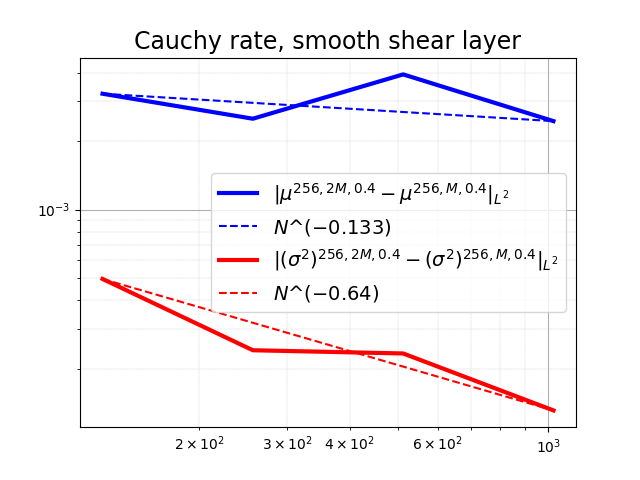}
		\subcaption{Smooth shear layer}
	\end{subfigure}%
	\begin{subfigure}{.4\textwidth}
		\centering
		\includegraphics[width=\linewidth]{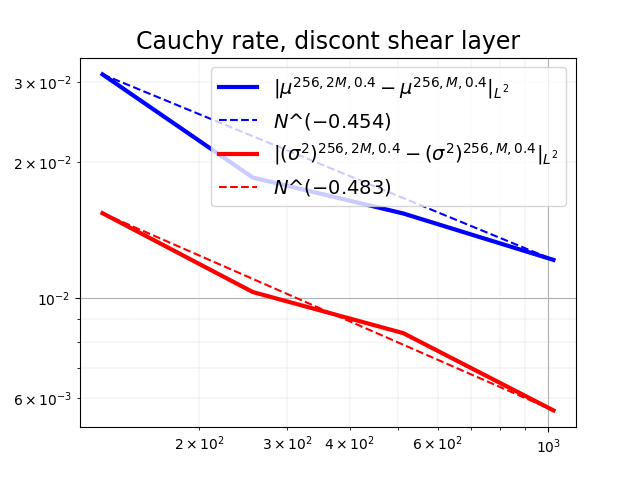}
		\subcaption{Discontinuous shear layer}
	\end{subfigure}
	\caption{Cauchy rates for mean and variance, for smooth and discontinuous shear layer, for $N=256$ and varying $M$, at $T=0.4$.}
	\label{fig:mcerror}
\end{figure}

\begin{figure}
	\centering
	\begin{subfigure}{.4\textwidth}
		\centering
		\includegraphics[width=\linewidth]{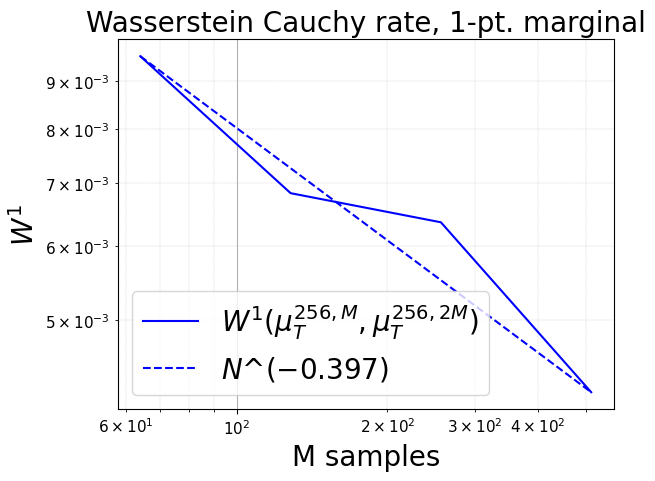}
		\subcaption{Smooth shear layer}
	\end{subfigure}%
	\begin{subfigure}{.4\textwidth}
		\centering
		\includegraphics[width=\linewidth]{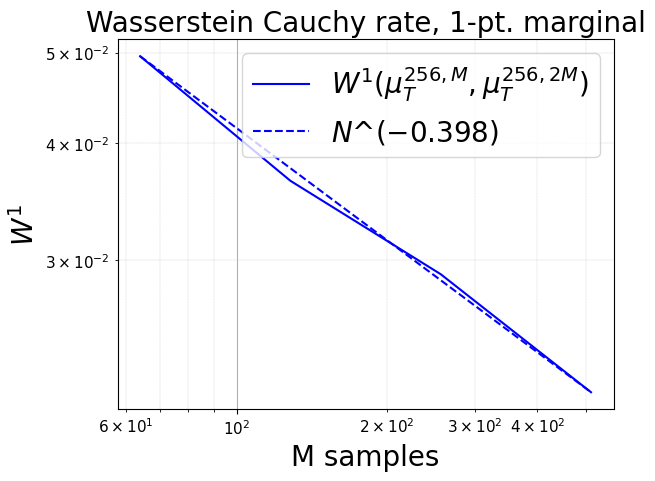}
		\subcaption{Discontinuous shear layer}
	\end{subfigure}

	\caption{Wasserstein distances for 1-point marginals for velocity, for smooth and discontinuous shear layer, for $N=256$ and varying $M$, at $T=0.4$.}
	\label{fig:mcwass}
\end{figure}

\subsection{Fractional Brownian motions}

A \emph{fractional Brownian motion}, \cite{Kolmogorov40}, \cite{Mandelbrot68}, with \emph{Hurst index} $H \in [0, 1)$ is a random noise field $X_H: \Omega \times D \to \R$, that takes normally-distributed values with mean zero at each point, and has covariance
\begin{equation*} \text{Cov}[X_H(x), X_H(y)] = \mathbb{E}\left[ X_H(x)X_H(y)\right] = \frac{1}{2}(\|x-y\|^{2H} - \|x\|^{2H} - \|y\|^{2H}). \end{equation*}

The case $H=0.5$ corresponds to a Brownian motion, often referred to as having ``independent increments''; this is not the case for $H \neq 0.5$. A realization of a Brownian motion with Hurst index $H$ is a.e. in $C^{0, \alpha}(D)$, with $ \alpha \le \max\{0, H - \epsilon \}$ for all $\epsilon > 0$, Prop. 6.4.2 in \cite{LeonardiPhd}. We display here that, even for low H\"older-regularity settings, assumption \eqref{hyp:scaling} holds. We generate realizations of fractional Brownian motions\footnote{More accurately, Brownian bridges, due to the periodic boundary conditions.} with the midpoint displacement method, \cite{Levy65}, \cite{LambertLacroix07}.

We present here results for $H \in \{0.75, 0.5, 0.15\}$; we include a realization of the initial data for each in Fig. \ref{fig:fbminit}.

\begin{figure}
	\centering
	\begin{subfigure}{.33\textwidth}
		\centering
		\includegraphics[width=\textwidth]{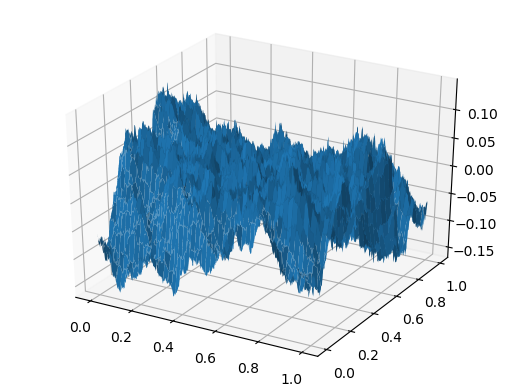}
	\end{subfigure}%
	\begin{subfigure}{.33\textwidth}
		\centering
		\includegraphics[width=\textwidth]{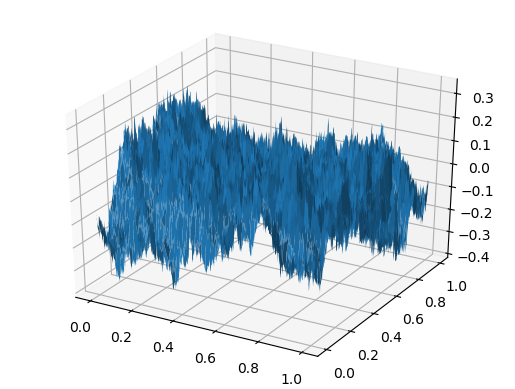}
	\end{subfigure}%
	\begin{subfigure}{.33\textwidth}
		\centering
		\includegraphics[width=\textwidth]{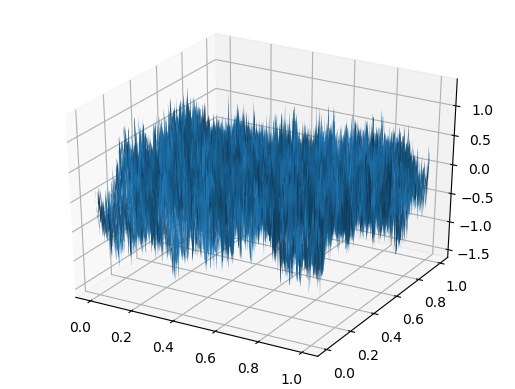}
	\end{subfigure}
	\caption{Three realizations of horizontal velocity for initial data in a fractional Brownian motion (left to right: $H=0.75$, $H=0.5$, $H=0.15$) after discretely divergence-free projection.}
	\label{fig:fbminit}
\end{figure}

We display in Fig. \ref{fig:fbmstruct} the structure function, for $t \in \{0, 1\}$, for all three cases. We observe, again, an initial exponent roughly corresponding to $H$, and a regularization effect as time advances. In Fig. \ref{fig:fbmslopes} we include the evolution of the fitted slope with time, clearly illustrating this.

\begin{figure}	
	\centering
	\begin{subfigure}{0.3\textwidth}
		\centering
		\includegraphics[width=\linewidth]{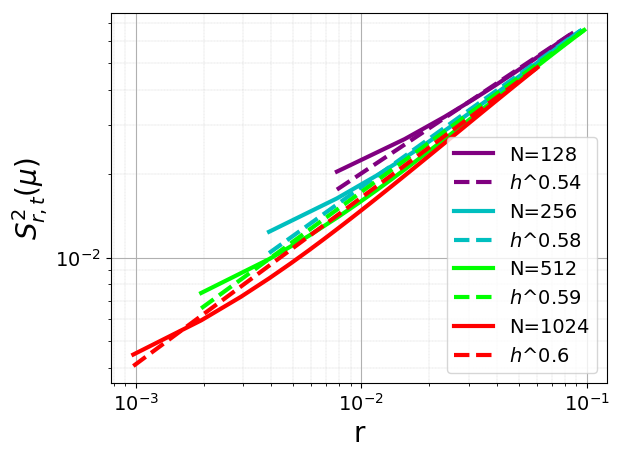}
	\end{subfigure} %
	\begin{subfigure}{0.3\textwidth}
		\centering
		\includegraphics[width=\linewidth]{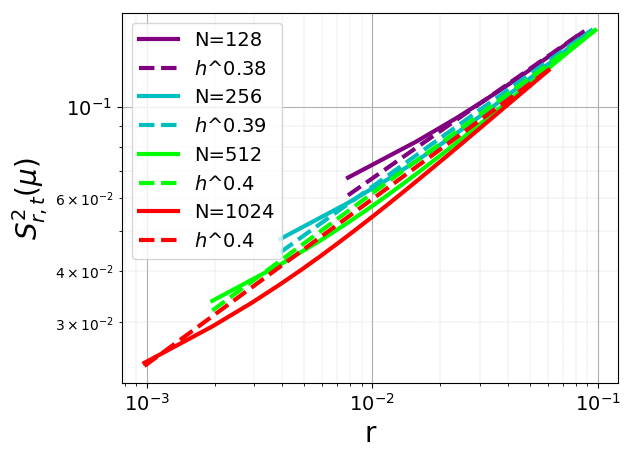}
	\end{subfigure} %
	\begin{subfigure}{0.3\textwidth}
		\centering
		\includegraphics[width=\linewidth]{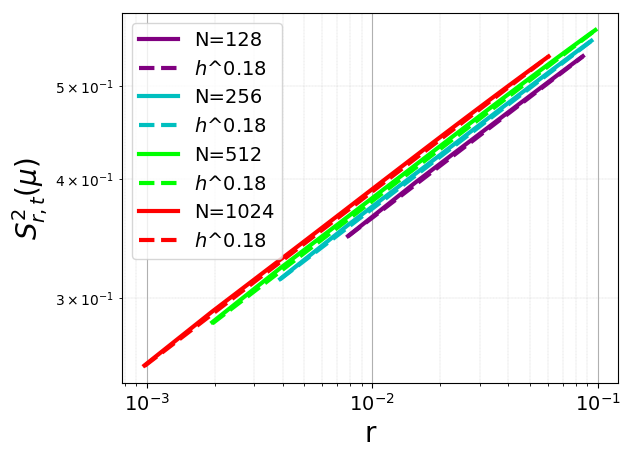}
	\end{subfigure}
	
	\begin{subfigure}{0.3\textwidth}
		\centering
		\includegraphics[width=\linewidth]{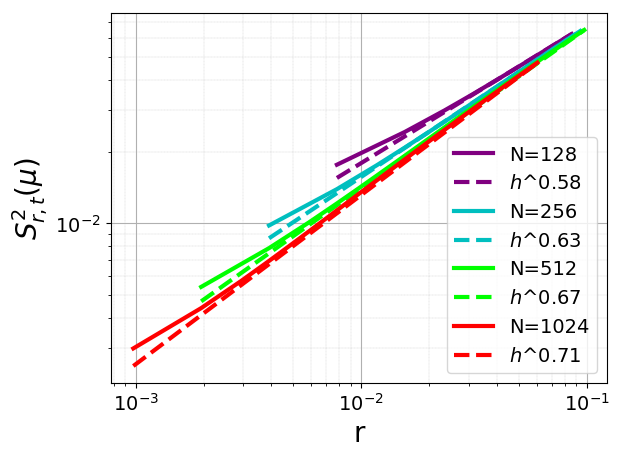}
	\end{subfigure} %
	\begin{subfigure}{0.3\textwidth}
		\centering
		\includegraphics[width=\linewidth]{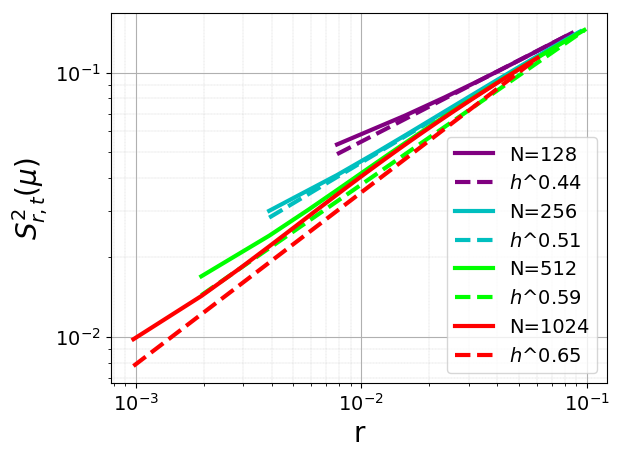}
	\end{subfigure} %
	\begin{subfigure}{0.3\textwidth}
		\centering
		\includegraphics[width=\linewidth]{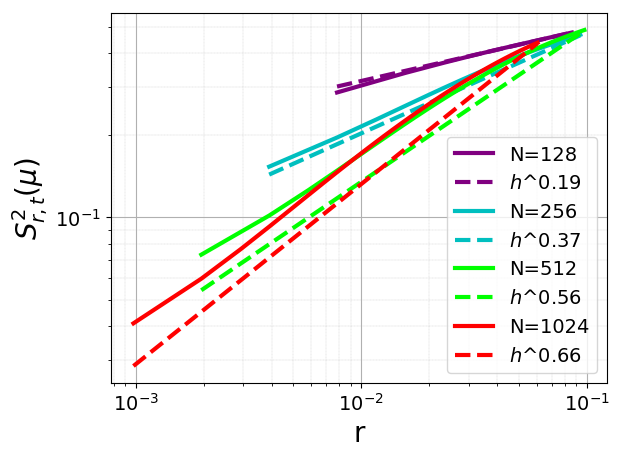}
	\end{subfigure}
	
	\caption{Structure functions for velocity vector, $p=2$, for fractional Brownian motion with (left to right) $H \in \{0.75, 0.5, 0.15\}$, for $t=0$ (top) and $t=1$ (bottom).}
	\label{fig:fbmstruct}
\end{figure}

\begin{figure}
	\centering
	\begin{subfigure}{.3\textwidth}
		\centering
		\includegraphics[width=\textwidth]{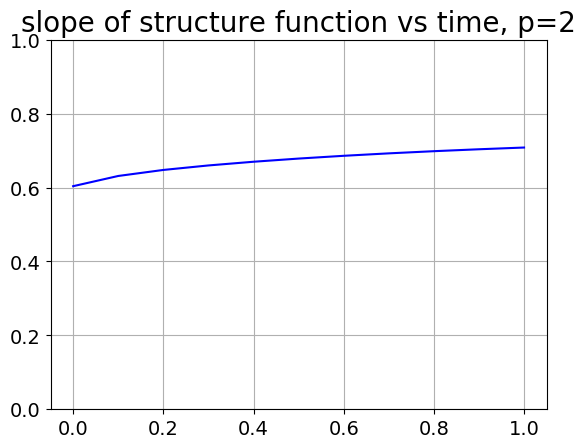}
	\end{subfigure}%
	\begin{subfigure}{.3\textwidth}
		\centering
		\includegraphics[width=\textwidth]{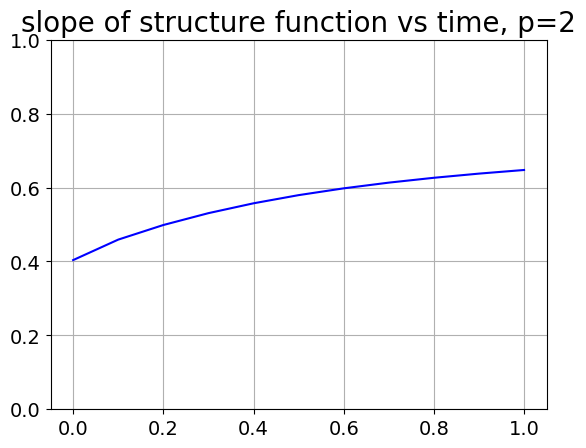}
	\end{subfigure}%
	\begin{subfigure}{.3\textwidth}
		\centering
		\includegraphics[width=\textwidth]{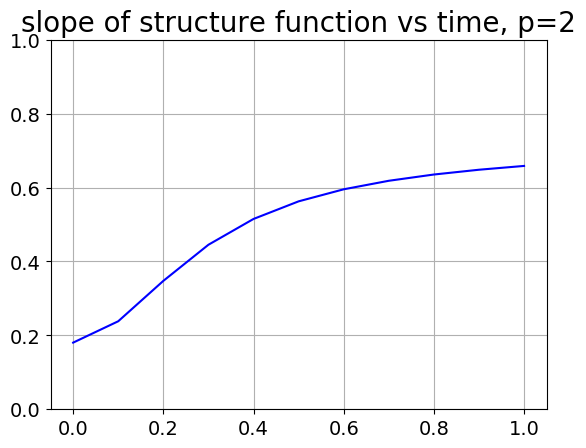}
	\end{subfigure}
	\vspace{0.5em}
	
	\caption{Time evolution of the best fit for the slope, as in Fig. \ref{fig:fbmstruct}, for the structure function with $p=2$. Left to right, $H=0.75$, $H=0.5$, $H=0.15$.}
	\label{fig:fbmslopes}
\end{figure}

Again, a uniform boundedness from above is strongly suggested, even for initial data in very low regularity spaces. Unsurprisingly again, we observe clear strong convergence of observables (Fig. \ref{fig:fbmmeanvar}) and vanishing Cauchy rates in Wasserstein metric, for the first few marginals, as in the previous test (Fig. \ref{fig:fbmwass}).

\begin{figure}
	\begin{subfigure}{0.31\textwidth}
		\centering
		\includegraphics[width=\linewidth]{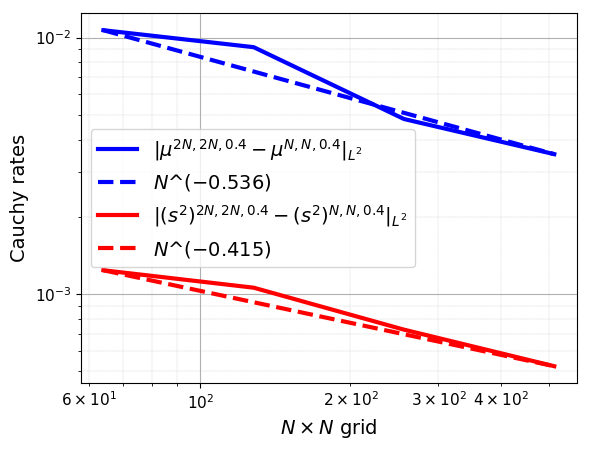}
	\end{subfigure} %
	\begin{subfigure}{0.31\textwidth}
		\centering
		\includegraphics[width=\linewidth]{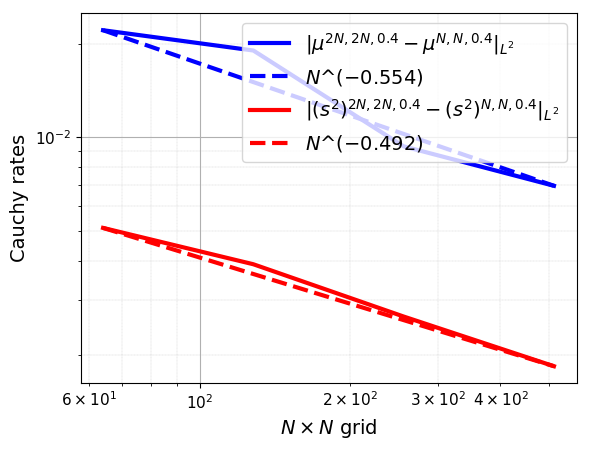}
	\end{subfigure} %
	\begin{subfigure}{0.31\textwidth}
		\centering
		\includegraphics[width=\linewidth]{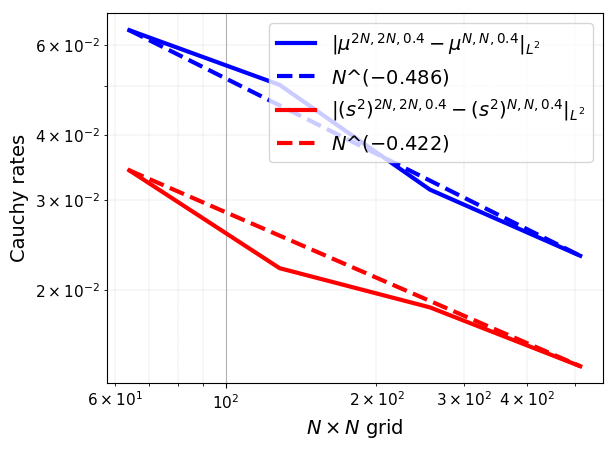}
	\end{subfigure} 
	
	\caption{Cauchy rates for mean and variance for fractional Brownian motions; left to right, $H=0.75$, $H=0.5$, $H=0.15$, at time $t=1$.}
	\label{fig:fbmmeanvar}
\end{figure}

\begin{figure}
	\begin{subfigure}{0.31\textwidth}
		\centering
		\includegraphics[width=\linewidth]{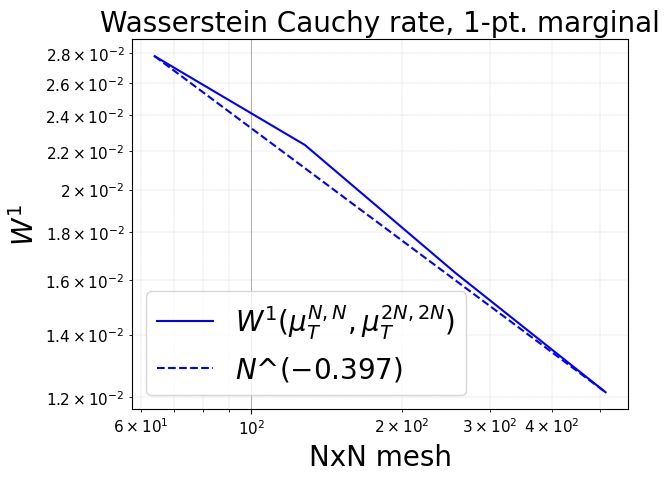}
	\end{subfigure} %
	\begin{subfigure}{0.31\textwidth}
		\centering
		\includegraphics[width=\linewidth]{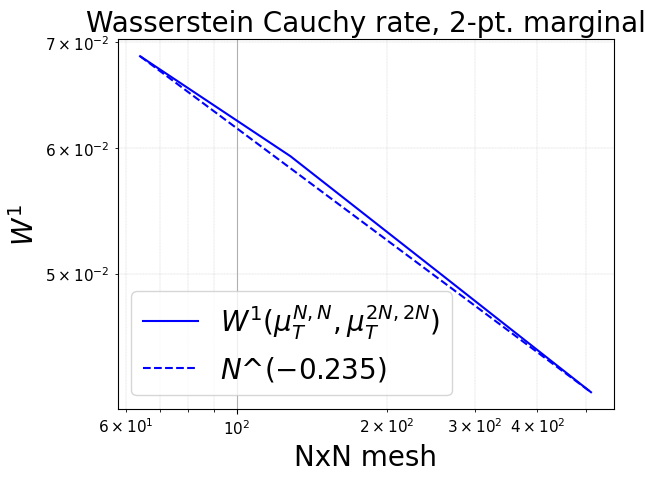}
	\end{subfigure} %
	\begin{subfigure}{0.31\textwidth}
		\centering
		\includegraphics[width=\linewidth]{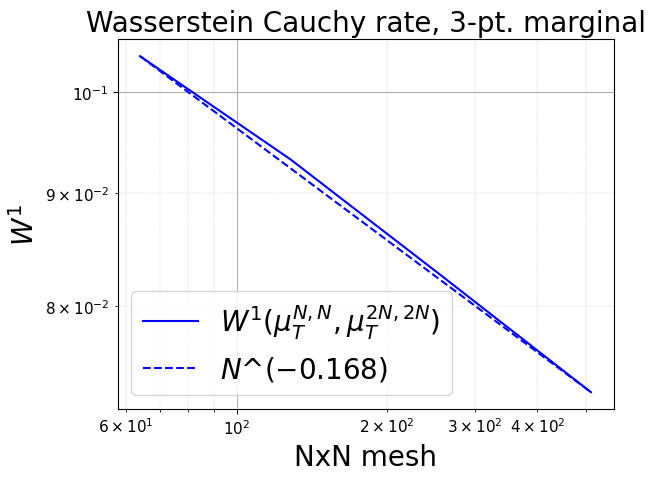}
	\end{subfigure} 
	
	\begin{subfigure}{0.31\textwidth}
		\centering
		\includegraphics[width=\linewidth]{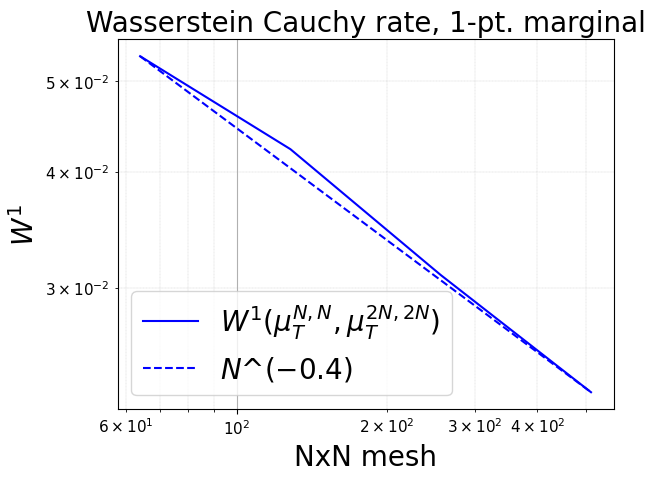}
	\end{subfigure} %
	\begin{subfigure}{0.31\textwidth}
		\centering
		\includegraphics[width=\linewidth]{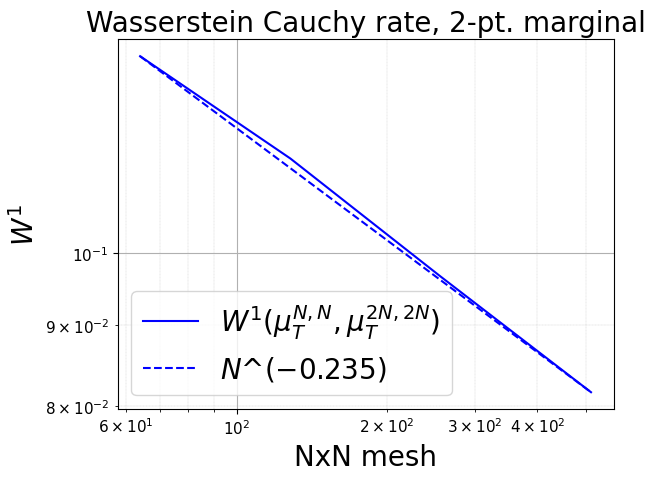}
	\end{subfigure} %
	\begin{subfigure}{0.31\textwidth}
		\centering
		\includegraphics[width=\linewidth]{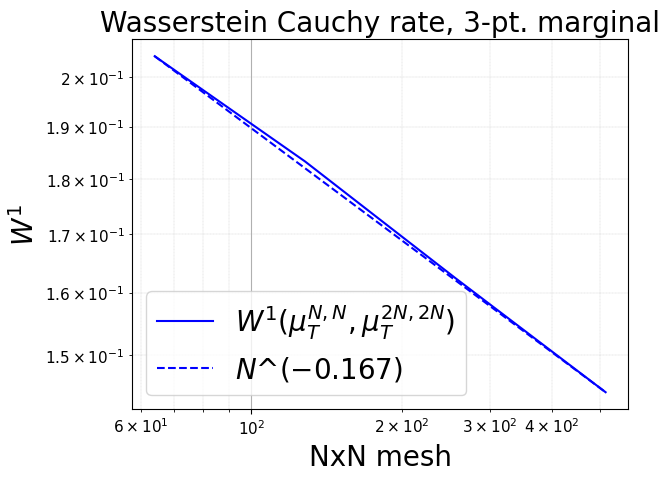}
	\end{subfigure} 
	
	\begin{subfigure}{0.31\textwidth}
		\centering
		\includegraphics[width=\linewidth]{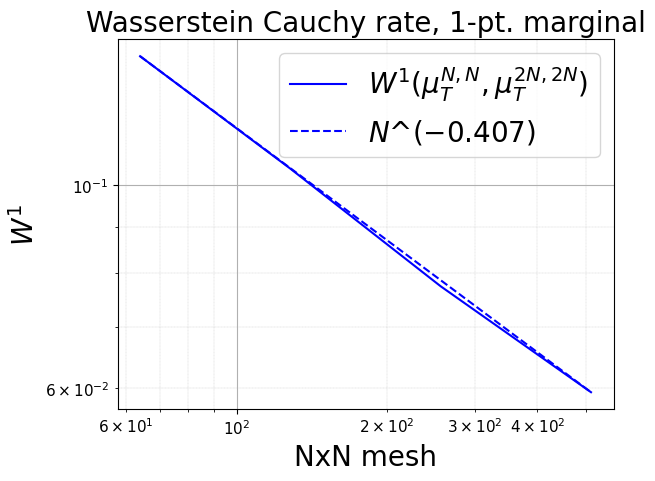}
	\end{subfigure} %
	\begin{subfigure}{0.31\textwidth}
		\centering
		\includegraphics[width=\linewidth]{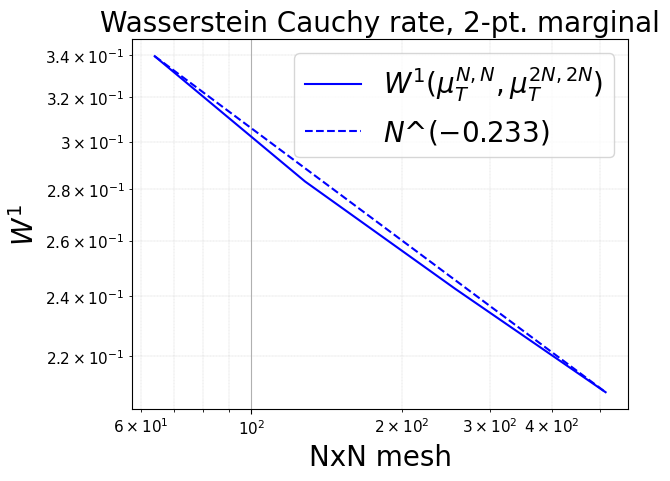}
	\end{subfigure} %
	\begin{subfigure}{0.31\textwidth}
		\centering
		\includegraphics[width=\linewidth]{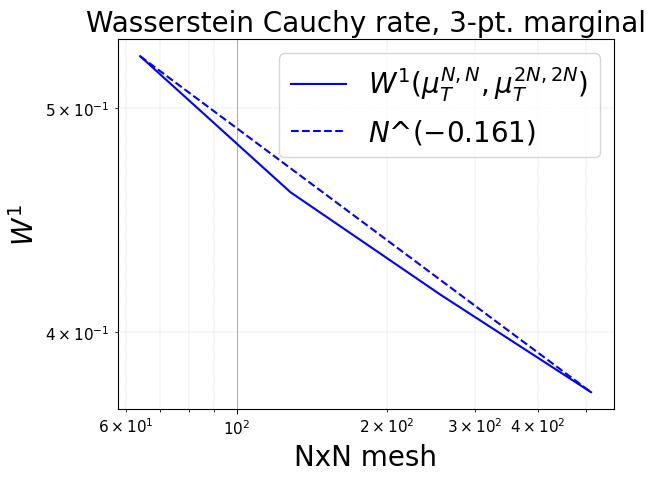}
	\end{subfigure}
	
	\caption{Cauchy rates for Wasserstein distances for the velocity vector field, (left to right) 1-, 2- and 3-point marginals at $t=1$. Top to bottom,  $H=0.75$, $H=0.5$, $H=0.15$.}
	\label{fig:fbmwass}
\end{figure}

\paragraph{A note on boundary conditions} The implementation of a finite volume solver on a Cartesian mesh for the incompressible fluid dynamics equations with arbitrary boundary conditions is non-trivial; see \cite{BCG89} and references therein; \cite{Gresho87} present an overview of classic techniques for the implementation of non-periodic boundary conditions. This treatment of the boundary terms would introduce a technical complication of the theory developed above. We remark, however, that the numerical solver of \cite{luqness} is able to reproduce some non-periodic boundary conditions. The following result is omitted for brevity, but in \cite{CPPThesis}, Section 6.3, the author presents an application of Algorithm \ref{algo:fkmt} to flow along a channel with ``no-flow'' (homogeneous Neumann) lids. The numerical results obtained are fully consistent with the theory observed, supporting our claim that periodicity of the boundary conditions is a technical assumption, as well as presenting, to our knowledge, the first instance in the literature of an approximate statistical solution of the incompressible Euler equations on a non-periodic domain.

\subsection{Stability}
In this paper, we have discussed only the version of Algorithm \ref{algo:fkmt} which uses the finite volume scheme \eqref{eq:scheme1}-\eqref{eq:scheme2} as an underlying method. In similar circumstances, e.g. \cite{Chen2012}, the authors express concern that different numerical solvers may converge to different (in their case, weak) solutions.

We believe that the framework of statistical solutions makes Algorithm \ref{algo:fkmt} robust with respect to the choice of the underlying solver. To this effect, we have run identical simulations with the following numerical schemes: the hyperviscosity spectral scheme of \cite{LMP1}, with the implementation in \cite{sphinx}; the high-order ENO finite difference scheme of \cite{MPP21}, and the high-order finite difference central scheme of \cite{Morinishi98}. For all of them, we observe similar results, and in fact the Wasserstein distance between the marginals computed with different schemes appears to tend to zero as the resolution is increased. See Fig. \ref{fig:stability} for some examples for the discontinuous shear layer; this is chosen for brevity as other examples display identical results.

\begin{figure}	
	\centering
	\begin{subfigure}{.35\textwidth}
		\centering
		\includegraphics[width=\linewidth]{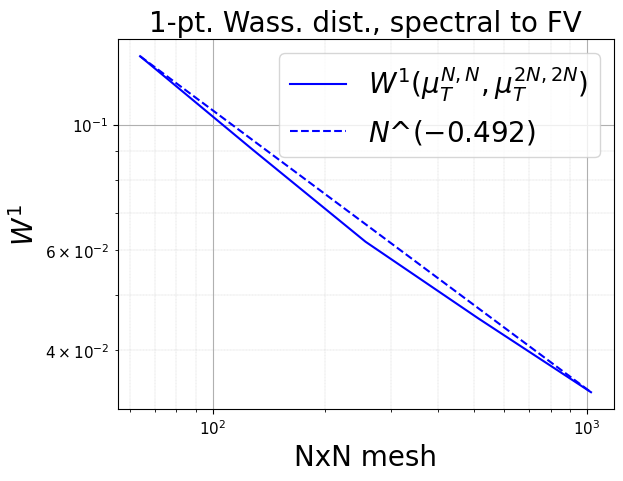}
	\end{subfigure} %
	\begin{subfigure}{.35\textwidth}
		\centering
		\includegraphics[width=\linewidth]{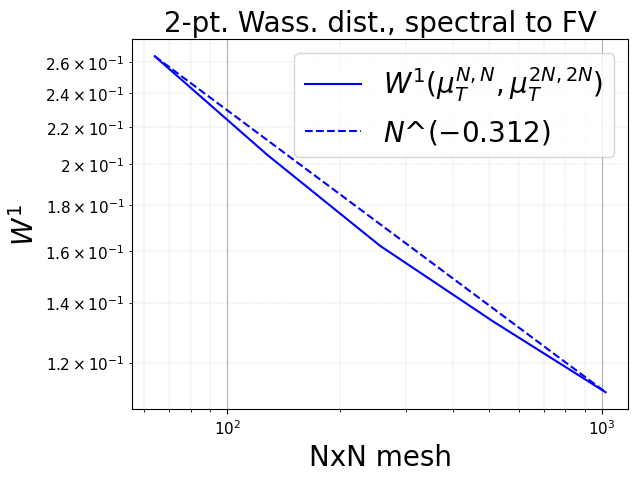}
	\end{subfigure}
	\vspace{0.5em}

	\begin{subfigure}{.35\textwidth}
		\centering
		\includegraphics[width=\linewidth]{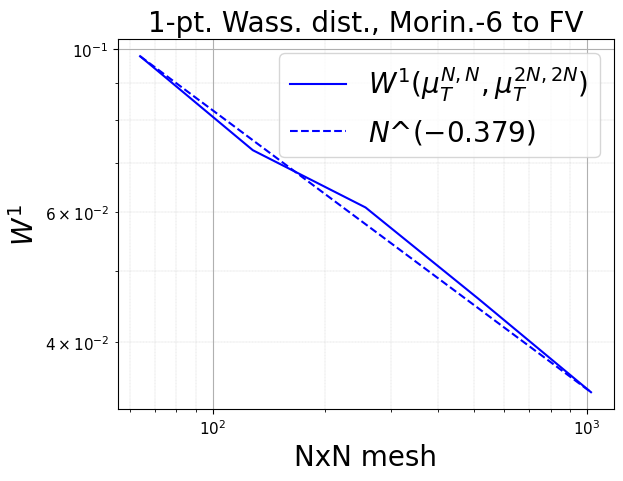}
	\end{subfigure} %
	\begin{subfigure}{.35\textwidth}
		\centering
		\includegraphics[width=\linewidth]{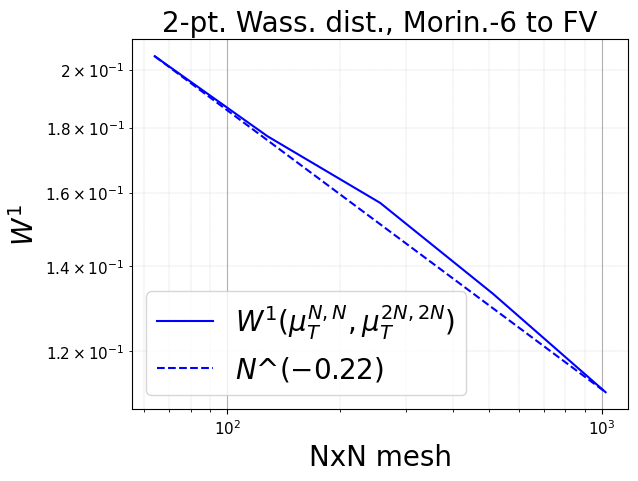}
	\end{subfigure}
	
	\caption{Wasserstein distance (left: one-point marginals; right: two-point marginals) between finite volume and hyperviscosity spectral solver of \cite{LMP1}, top; and between finite volume and 6th order scheme of \cite{Morinishi98}, bottom. Discontinuous vortex shear layer. Distance is measured at the same resolution and number of samples for both solvers, $64\times 64$ to $1024 \times 1024$.}
	\label{fig:stability}
\end{figure}

Furthermore, we also remark that for the case of deterministic initial data, we have suggested an arbitrary perturbation. Numerical experiments suggest that the choice of perturbation is, as the magnitude goes to zero, of little practical consequence.

\section{Conclusions and future work}
In this work we have presented an algorithm for the efficient computation of statistical solutions, rigorously proven its convergence, and displayed its practical use.

As stated in Theorem \ref{thm:relativeconv}, this convergence is only up to a subsequence, and conditioned to an external \emph{scaling assumption} on the structure function. However, in section \ref{sec:examples}, we have shown some numerical experiments that give credibility to the claim that this assumption is relatively mild, and in two dimensions it holds for every numerical experiment considered.

Furthermore, we find that the results obtained are fully consistent with those found in the literature, e.g. \cite{LMP1} for spectral methods. This supports the idea that a Monte Carlo algorithm, presented here as \ref{algo:fkmt}, can be non-intrusively applied over any standard scheme in the literature. 

In particular, the crucial novel result of this work is that a finite volume scheme can be used for this goal. To the knowledge of the author, the only analogous previously existing results are those of \cite{LMP1}, for spectral schemes. These, although highly efficient, are suitable only for a very limited selection of problems, in simple domains with periodic boundary conditions. Although the results presented here are derived in the torus, this is a technical assumption for convenience; there is every reason to believe they hold analogously for other types of periodic boundary conditions. In fact, a numerical example with homogeneous Neumann boundary conditions was presented by the author in \cite{CPPThesis}, where in particular the scaling assumption holds as well.

Algorithm \ref{algo:fkmt} presented here is based on a relatively coarse Monte Carlo integration. Techniques such as multi-level Monte Carlo methods, quasi Monte Carlo, etc. could be applied to accelerate the convergence and reduce the amount of high-resolution samples needed. Some attempts exist in the literature at applying these, e.g. \cite{LeonardiPhd}, \cite{LyePhd}, but their success has been limited. Analogously, in \cite{CPPThesis} we study the possible application of high-order schemes, in space and time; we find that these do not appear to produce significant gains.

This strongly suggests that, if the rate of the convergence is to be improved, the bottleneck is in the Monte Carlo sampling. In fact, we have only presented two-dimensional numerical experiments here. The theory, as derived here, holds for higher dimensions as well; the limitation to two-dimensional problems is due to computational constraints. There is presently ongoing work, \cite{RohnerMsc}, to develop an efficient, highly-parallelizable three-dimensional solver to carry out the corresponding experiments.

Nonetheless, we believe that this work extends the state of the art and provides consistent evidence that a Monte Carlo scheme can be combined with any well-known numerical method for the incompressible Euler equations to produce statistical solutions. Similar conclusions have been reached in e.g. \cite{FLMWSystems} for systems of conservation laws (including the compressible Euler equations), furthering the credibility that statistical solutions, approximated through Monte Carlo integration, are a satisfactory technique for the solution of partial differential equations, both as a sound theoretical framework and as a practically tractable approach.

\bibliographystyle{agsm}
\bibliography{bib_thesis}

@PHDTHESIS{LeonardiPhd,
	author = {Leonardi, Filippo},
	publisher = {ETH Zurich},
	year = {2018},
	language = {en},
	copyright = {In Copyright - Non-Commercial Use Permitted},
	institution = {EC and EC},
	size = {153 p.},
	DOI = {10.3929/ethz-b-000296003},
	title = {Numerical methods for ensemble based solutions to incompressible flow equations},
	school = {ETH Zurich}
}

@PHDTHESIS{LyePhd,
	author = {Lye, Kjetil O.},
	publisher = {ETH Zurich},
	year = {2020},
	language = {en},
	DOI = {https://doi.org/10.3929/ethz-b-000432014},
	title = {Computation of statistical solutions of hyperbolic systems of conservation laws},
	school = {ETH Zurich}
}

@article{FLMWSystems,
	title={Statistical solutions of hyperbolic systems of conservation laws: Numerical approximation},
	author={Fjordholm, Ulrik S. and Lye, Kjetil and Mishra, Siddhartha and Weber, Franziska},
	journal={Mathematical Models and Methods in Applied Sciences},
	volume={30},
	number={03},
	pages={539--609},
	year={2020},
	publisher={World Scientific}
}

@article{FMTActa,
	 title={On the computation of measure-valued solutions}, 
	 volume={25}, 
	 DOI={10.1017/S0962492916000088}, 
	 journal={Acta Numerica}, 
	 publisher={Cambridge University Press}, 
	 author={Fjordholm, Ulrik S. and Mishra, Siddhartha and Tadmor, Eitan}, 
	 year={2016}, 
	 ages={567--679}
 }

@article{FeffermanClay,
	author = {Fefferman, Charles},
	year = {2006},
	month = {01},
	pages = {},
	title = {Existence and smoothness of the {Navier-Stokes} equation},
	journal = {The Millennium Prize Problems}
}

@article{WiedemannExistence,
	title = {Existence of weak solutions for the incompressible {E}uler equations},
	journal = {Annales de l'Institut Henri Poincaré C, Analyse non linéaire},
	volume = {28},
	number = {5},
	pages = {727--730},
	year = {2011},
	issn = {0294-1449},
	doi = {https://doi.org/10.1016/j.anihpc.2011.05.002},
	author = {Wiedemann, Emil}
}

@article{LMP1,
		title = {Statistical solutions of the incompressible {E}uler equations},
		author = {Lanthaler, Samuel and Mishra, Siddhartha and Par\'{e}s-Pulido, Carlos},
		journal = {Mathematical Models and Methods in Applied Sciences},
		volume = {31},
		issue = {2},
		pages = {223--292},
		year = {2021},
		doi = {10.1142/S0218202521500068}
	}

@book{MajdaBertozzi, place={Cambridge}, series={Cambridge Texts in Applied Mathematics}, title={Vorticity and Incompressible Flow}, DOI={10.1017/CBO9780511613203}, publisher={Cambridge University Press}, author={Majda, Andrew J. and Bertozzi, Andrea L.}, year={2001}, collection={Cambridge Texts in Applied Mathematics}}

@book{Ladyzhenskaya69,
	title = {The Mathematical Theory of Viscous Incompressible Flows},
	edition = {2nd edition},
	author = {Ladyzhenskaya, Olga},
	year = {1969},
	publisher = {Gordon and Breech}
}

@article{BealeKatoMajda,
	author = {Beale, J. Thomas and Kato, Tosio and Majda, Andrew},
	title = {{Remarks on the breakdown of smooth solutions for the $3$-D Euler equations}},
	volume = {94},
	journal = {Communications in Mathematical Physics},
	number = {1},
	publisher = {Springer},
	pages = {61--66},
	year = {1984},
	doi = {cmp/1103941230},
}

@article{Yudovich63,
	title = {Non-stationary flow of an ideal incompressible liquid},
	journal = {USSR Computational Mathematics and Mathematical Physics},
	volume = {3},
	number = {6},
	pages = {1407--1456},
	year = {1963},
	issn = {0041-5553},
	doi = {https://doi.org/10.1016/0041-5553(63)90247-7},
	author = {Yudovich, Viktor I.}
}

@article{Delort91,
	author = {Delort, Jean-Marc},
	journal = {J. Am. Math. Soc.},
	number = {3},
	pages = {553--586},
	title = {Existence de Nappes de Tourbillon en Dimension Deux},
	volume = {4},
	year = {1991}
}

@Article{Vecchi93,
	author={Vecchi, Italo and Wu, Sijue},
	title={On {L1}-vorticity for 2-{D} incompressible flow},
	journal={Manuscripta Mathematica},
	year={1993},
	month={Dec},
	day={01},
	volume={78},
	number={1},
	pages={403--412},
	issn={1432-1785},
	doi={10.1007/BF02599322}
}

@article{Leray34,
	author = {Leray, Jean},
	year = {1934},
	title = {Sur le mouvement d'un liquide visqueux emplissant l'espace},
	journal = {Acta Mathematica},
	volume = {63},
	pages = {193--248},
	doi = {0.1007/BF02547354}
}

@article{DeLellisSz09,
	ISSN = {0003486X},
	author = {De Lellis, Camillo and Sz\'{e}kelyhidi, Jr, L\'{a}zsl\'{o}},
	journal = {Annals of Mathematics},
	number = {3},
	pages = {1417--1436},
	publisher = {Annals of Mathematics},
	title = {The {E}uler equations as a differential inclusion},
	volume = {170},
	year = {2009}
}

@Article{DeLellisSz2013,
	author={De Lellis, Camillo
	and Sz{\'e}kelyhidi, Jr., L{\'a}szl{\'o}},
	title={Dissipative continuous {E}uler flows},
	journal={Inventiones mathematicae},
	year={2013},
	month={Aug},
	day={01},
	volume={193},
	number={2},
	pages={377-407},
	issn={1432-1297},
	doi={10.1007/s00222-012-0429-9},
}

@article{Szekelyhidi11,
	title = {Weak solutions to the incompressible {E}uler equations with vortex sheet initial data},
	journal = {Comptes Rendus Mathematique},
	volume = {349},
	number = {19},
	pages = {1063-1066},
	year = {2011},
	issn = {1631-073X},
	doi = {https://doi.org/10.1016/j.crma.2011.09.009},
	author = {Sz{\'{e}}kelyhidi, Jr, L\'{a}zsl\'{o}}
}

@book{Lions96,
	title={Mathematical Topics in Fluid Mechanics: Vol. 1-2},
	author={Lions, Pierre-Louis},
	isbn={9780198514886},
	lccn={2013444083},
	series={Mathematical Topics in Fluid Mechanics},
	year={1996},
	publisher={Clarendon Press}
}

@Article{DiPerna85,
	author={DiPerna, Ronald J},
	title={Measure-valued solutions to conservation laws},
	journal={Archive for Rational Mechanics and Analysis},
	year={1985},
	month={Sep},
	day={01},
	volume={88},
	number={3},
	pages={223-270},
	issn={1432-0673},
	doi={10.1007/BF00752112}
}

@Article{FLMfoundations,
	author={Fjordholm, Ulrik S.	and Lanthaler, Samuel and Mishra, Siddhartha},
	title={Statistical Solutions of Hyperbolic Conservation Laws: Foundations},
	journal={Archive for Rational Mechanics and Analysis},
	year={2017},
	month={Nov},
	day={01},
	volume={226},
	number={2},
	pages={809-849},
	issn={1432-0673},
	doi={10.1007/s00205-017-1145-9}
}

@Article{FKMT17,
	author={Fjordholm, Ulrik S.
	and K{\"a}ppeli, Roger
	and Mishra, Siddhartha
	and Tadmor, Eitan},
	title={Construction of Approximate Entropy Measure-Valued Solutions for Hyperbolic Systems of Conservation Laws},
	journal={Foundations of Computational Mathematics},
	year={2017},
	month={Jun},
	day={01},
	volume={17},
	number={3},
	pages={763-827},
	issn={1615-3383},
	doi={10.1007/s10208-015-9299-z}
}

@Article{Lichtenstein25,
	author={Lichtenstein, Leon},
	title={{\"U}ber einige {E}xistenzprobleme der {H}ydrodynamik homogener, unzusammendr{\"u}ckbarer, reibungsloser {F}l{\"u}ssigkeiten und die {H}elmholtzschen {W}irbels{\"a}tze},
	journal={Mathematische Zeitschrift},
	year={1925},
	month={Dec},
	day={01},
	volume={23},
	number={1},
	pages={89-154},
	issn={1432-1823},
	doi={10.1007/BF01506223}
}

@Article{DiPerna87,
	author={DiPerna, Ronald J.
	and Majda, Andrew J.},
	title={Oscillations and concentrations in weak solutions of the incompressible fluid equations},
	journal={Communications in Mathematical Physics},
	year={1987},
	month={Dec},
	day={01},
	volume={108},
	number={4},
	pages={667-689},
	issn={1432-0916},
	doi={10.1007/BF01214424}
}

@Article{Brenier11,
	author={Brenier, Yann and De Lellis, Camillo and Sz\'{e}kelyhidi, L\'{a}szl\'{o}},
	title = {Weak-strong uniqueness for measure-valued solutions},
	year={2011},
	journal={Communications in Mathematical Physics},
	volume={305},
	pages={351--361},
	publisher={Springer}
}

@article{Feireisl19,
	title={$\mathcal{K}$-convergence as a new tool in numerical analysis},
	author={Feireisl, Eduard and Luk{\'a}{\v{c}}ov{\'a}-Medvi{\v{d}}ov{\'a}, M{\'a}ria and Mizerov{\'a}, Hana},
	journal={IMA Journal of Numerical Analysis},
	volume={40},
	number={4},
	pages={2227--2255},
	year={2020},
	publisher={Oxford University Press}
}

@article{BCG89,
	title = {A second-order projection method for the incompressible {N}avier-{S}tokes equations},
	journal = {Journal of Computational Physics},
	volume = {85},
	number = {2},
	pages = {257-283},
	year = {1989},
	issn = {0021-9991},
	author = {Bell, John B. and Colella, Phillip and Glaz, Harland M.},
}

@Article{Leonardi18,
	author={Leonardi, Filippo},
	title={A projection method for the computation of admissible measure valued solutions of the incompressible {E}uler equations},
	journal={Discrete {\&} Continuous Dynamical Systems - S},
	year={2018},
	month={Jun},
	day={1},
	volume={11},
	number={5},
	pages={941-961},
	doi={10.3934/dcdss.2018056}
}

@article{Epstein69,
	author = {Epstein, Edward S.},
	year  = {1969},
	title = {Stochastic dynamic prediction},
	journal = {Tellus},
	volume = {21},
	number = {6},
	pages = {739-759},
	publisher = {Taylor & Francis},
	doi = {10.3402/tellusa.v21i6.10143},
	eprint = { https://doi.org/10.3402/tellusa.v21i6.10143	}
}

@Article{Leith74,
	author={Leith, Cecil E.},
	title={Theoretical Skill of {M}onte {C}arlo Forecasts},
	journal={Monthly Weather Review},
	year={01 Jun. 1974},
	publisher={American Meteorological Society},
	address={Boston MA, USA},
	volume={102},
	number={6},
	pages={409-418},
	doi={10.1175/1520-0493(1974)102<0409:TSOMCF>2.0.CO;2},
	language={English}
}

@misc{luqness,
	author = {Leonardi, Filippo},
	title = {lUQness},
	year = {2017},
	publisher = {GitLab},
	journal = {GitLab repository},
	howpublished = {https://gitlab.com/sam-uq/luqness}
}

@misc{sphinx,
	author = {Leonardi, Filippo},
	title = {Sphinx},
	year = {2017},
	publisher = {GitLab},
	journal = {GitLab repository},
	howpublished = {https://gitlab.com/sam-uq/sphinx}
}

@Article{MPP21,
	author = {Mishra, Siddhartha and Par\'{e}s-Pulido, Carlos and Pressel, Kyle G.},
	title = {Arbitrarily High-Order (Weighted) Essentially Non-Oscillatory Finite Difference Schemes for Anelastic Flows on Staggered Meshes},
	journal = {Communications in Computational Physics},
	year = {2021},
	volume = {29},
	number = {5},
	pages = {1299--1335},
	issn = {1991-7120},
	doi = {https://doi.org/10.4208/cicp.OA-2020-0046}
}

@article{Morinishi98,
	author = {Morinishi, Yohei and Lund, T.S. and Vasilyev, Oleg V. and Moin, Parviz},
	doi = {10.1006/jcph.1998.5962},
	journal = {Journal of Computational Physics},
	number = {1},
	pages = {90--124},
	publisher = {Academic Press},
	title = {{Fully Conservative Higher Order Finite Difference Schemes for Incompressible Flow}},
	volume = {143},
	year = {1998}
}

@Article{Mandelbrot68,
	author={Mandelbrot, Benoit B.
	and Van Ness, John W.},
	title={Fractional {B}rownian Motions, Fractional Noises and Applications},
	journal={SIAM Review},
	year={1968},
	month={2021/04/12/},
	publisher={Society for Industrial and Applied Mathematics},
	volume={10},
	number={4},
	pages={422-437},
	issn={00361445}
}

@book{Levy65,
	author = {L{\'e}vy, Paul},
	isbn = {2-87647-091-8},
	mrclass = {60-02 (60G05 60G12 60J65)},
	mrnumber = {1188411},
	note = {Followed by a note by {M}. {L}o{\`e}ve, Reprint of the second (1965) edition},
	pages = {iv+437},
	publisher = {{\'E}ditions Jacques Gabay, Sceaux},
	series = {Les Grands Classiques Gauthier-Villars. [Gauthier-Villars Great Classics]},
	title = {Processus stochastiques et mouvement {B}rownien},
	year = {1992}
}

@article{Kolmogorov40,
	title={{W}ienersche {S}piralen und einige andere interessante {K}urven in {H}ilbertschen {R}aum},
	author={Kolmogorov, Andrei N.},
	journal={Acad. Sci. URSS (NS)},
	volume={26},
	pages={115--118},
	year={1940}
}

@article{LambertLacroix07,
	author = {Lambert-Lacroix, Sophie and Istas, Jacques and Brouste, Alexandre},
	year = {2007},
	month = {11},
	pages = {},
	title = {On Fractional {G}aussian Random Fields Simulations},
	volume = {23},
	journal = {Journal of Statistical Software}
}

@article{Gresho87,
	author = {Gresho, Philip M. and Sani, Robert L.},
	title = {On pressure boundary conditions for the incompressible {N}avier--{S}tokes equations},
	journal = {International Journal for Numerical Methods in Fluids},
	volume = {7},
	number = {10},
	pages = {1111-1145},
	doi = {https://doi.org/10.1002/fld.1650071008},
	eprint = {https://onlinelibrary.wiley.com/doi/pdf/10.1002/fld.1650071008},
	year = {1987}
}

@article{FLM18,
	title={Numerical approximation of statistical solutions of scalar conservation laws},
	author={Fjordholm, Ulrik S. and Lye, Kjetil and Mishra, Siddhartha},
	journal={SIAM Journal on Numerical Analysis},
	volume={56},
	number={5},
	pages={2989--3009},
	year={2018},
	publisher={SIAM}
}

@mastersthesis{RohnerMsc,
	author = {Rohner, Tobias},
	type = {Master thesis},
	title = {{E}fficient implementation of spectral viscosity methods for the incompressible {E}uler equations (provisional title)},
	school = {ETH Zurich},
	year = {2021}
}

@article{Buckmaster2019,
	title={Nonuniqueness of weak solutions to the {N}avier-{S}tokes equation},
	author={Buckmaster, Tristan and Vicol, Vlad},
	journal={Annals of {M}athematics},
	volume={189},
	number={1},
	pages={101--144},
	year={2019}
}

@article{Foias1973,
	title={Statistical study of {N}avier-{S}tokes equations, {II}},
	author={Foias, Ciprian},
	journal={Rendiconti del Seminario Matematico della Universita di Padova},
	volume={49},
	pages={9--123},
	year={1973}
}

@inproceedings{Foias2013,
	title={Properties of time-dependent statistical solutions of the three-dimensional {N}avier-{S}tokes equations},
	author={Foias, Ciprian and Rosa, Ricardo and Temam, Roger},
	booktitle={Annales de l'Institut Fourier},
	volume={63(6)},
	pages={2515--2573},
	year={2013}
}

@article{Vishik1979,
	title={Some mathematical problems of statistical hyromechanics},
	author={Vishik, Marko I. and Komech, Aleksandr I. and Fursikov, Andrei V.},
	journal={Russian Mathematical Surveys},
	volume={34},
	number={5},
	pages={149--234},
	year={1979}
}

@article{Langtangen2002,
	title={Numerical methods for incompressible viscous flow},
	author={Langtangen, Hans P. and Mardal, Kent-Andre and Winther, Ragnar},
	journal={Advances in water Resources},
	volume={25},
	number={8-12},
	pages={1125--1146},
	year={2002},
	publisher={Elsevier}
}

@article{LM15,
	title={Computation of measure-valued solutions for the incompressible {E}uler equations},
	author={Lanthaler, Samuel and Mishra, Siddhartha},
	journal={Mathematical Models and Methods in Applied Sciences},
	volume={25},
	number={11},
	pages={2043--2088},
	year={2015},
	publisher={World Scientific}
}

@PHDTHESIS{CPPThesis,
	author = {Par\'es-Pulido, Carlos},
	publisher = {ETH Zurich},
	year = {2021},
	language = {en},
	title = {Statistical solutions for the incompressible {E}uler equations with finite volume methods},
	school = {ETH Zurich}
}

@book{Frisch,
	title={Turbulence: the legacy of {A}. {N}. Kolmogorov},
	author={Frisch, Uriel and Kolmogorov},
year={1995},
publisher={Cambridge University Press}
}

@article{Chen2012,
	title={{K}olmogorov’s Theory of Turbulence and Inviscid Limit of the {N}avier-{S}tokes Equations in $\mathbb{R}^3$},
	author={Chen, Gui-Qiang and Glimm, James},
	journal={Communications in Mathematical Physics},
	volume={310},
	number={1},
	pages={267--283},
	year={2012},
	publisher={Springer}
}

\end{document}